\documentclass[12pt,leqno]{amsart}
\usepackage{amsmath,amssymb,amsfonts}
\usepackage{eucal,enumerate}
\usepackage[all]{xy}
\usepackage{graphicx,psfrag}

\newcommand\vanish[1]{}

\newtheorem{lem}{Lemma}[section]
\newtheorem{cor}[lem]{Corollary}
\newtheorem{prop}[lem]{Proposition}
\newtheorem{thm}[lem]{Theorem}
  \newtheorem{exama}{Example} [section]
\newenvironment{exam}{\begin{exama}\rm}{\end{exama}}
  \newtheorem{conja}{Conjecture}[section]
\newenvironment{conj}{\begin{conja}\rm}{\end{conja}}
  \newtheorem{questa}{Question}[section]
\newenvironment{question}{\begin{questa}\rm}{\end{questa}}

\numberwithin{equation}{section}
\numberwithin{figure}{section}
\numberwithin{table}{section}

\allowdisplaybreaks

\renewcommand{\phi}{\varphi}                 
\renewcommand{\epsilon}{\varepsilon}
\newcommand\eset{\varnothing}
\newcommand\setm{\setminus}
\newcommand\ub\underbar

\newcommand\vstrut[1]{\rule{0ex}{#1}}

\newcommand\cupdot {\mbox{\hspace{.15em}$\cup$\hspace{-.47em}$\cdot$\hspace{.4em}}} 
\newcommand \id{\mathrm{id}}
\newcommand\del{\nabla}
\newcommand\inv{^{-1}}
\newcommand\Aut{\operatorname{Aut}}
\newcommand\SwAut{\operatorname{SwAut}}
\newcommand\Sw{\operatorname{Sw}}
\newcommand\signs{\{+,-\}}

\newcommand\cC{{\mathcal C}}

\newcommand \fA{\mathfrak A}

\newcommand \fD{\mathfrak D}
\newcommand \fG{\mathfrak G}
\newcommand \fH{\mathfrak H}
\newcommand \fK{\mathfrak K}
\newcommand \fS{\mathfrak S}
\newcommand \fV{\mathfrak V}
\newcommand\fZ{\mathcal{Z}}

\newcommand\barp{\bar p}
\newcommand\eps{\epsilon}
\newcommand\bareps{\bar\epsilon}
\newcommand\5{\{1,2,3,4,5\}}
\newcommand\4{\{1,2,3,4\}}
\newcommand\3{\{1,2,3\}}
\newcommand\efg{\{e,f,g\}}
\newcommand\ef{\{e,f\}}

\begin{document}

\begin{center}
{\Large
Six Signed Petersen Graphs, and their Automorphisms
}
\vspace{0.5cm}

{\large
Thomas Zaslavsky}\\[0.5cm]
Department of Mathematical Sciences, \\ Binghamton University (SUNY), \\ Binghamton, NY 13902-6000, U.S.A. \\ E-mail: {\tt zaslav@math.binghamton.edu} \\ \today

\end{center}

\bigskip\hrule\bigskip

\vspace{0.5cm}


\vspace{0.5cm}    


\begin{quote}
{
\emph{Abstract.}
Up to switching isomorphism there are six ways to put signs on the edges of the Petersen graph.  We prove this by computing switching invariants, especially frustration indices and frustration numbers, switching automorphism groups, chromatic numbers, and numbers of proper 1-colorations, thereby illustrating some of the ideas and methods of signed graph theory.  We also calculate automorphism groups and clusterability indices, which are not invariant under switching.  In the process we develop new properties of signed graphs, especially of their switching automorphism groups.
}
\end{quote}

\setcounter{tocdepth}{1}
\tableofcontents

\section{Introduction}\label{intro}

The Petersen graph $P$ is a famous example and counterexample in graph theory, making it an appropriate subject for a book (see \cite{TPG}).  With signed edges it makes a fascinating example of many aspects of signed graph theory as well.  There are $2^{15}$ ways to put signs on the edges of $P$, but in many respects only six of them are essentially different.  We show how and why that is true as we develop basic properties of these six signed Petersens.

\begin{figure}[htbp]
\begin{center}
\includegraphics[scale=.6]{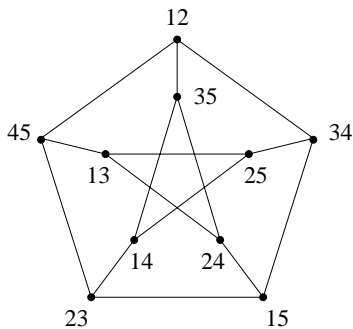}
\caption{$P$, the Petersen graph.}
\label{F:P}
\end{center}
\end{figure}

The fundamental property of signed graphs is balance.  A signed graph is \emph{balanced} if all its circles (circuits, cycles, polygons) have positive sign product.  
Harary introduced signed graphs and balance \cite{NB} (though they were implicit in K\"onig \cite[\S X.3]{Konig}).  Cartwright and Harary used them to model social stress in small groups of people in social psychology \cite{CH}.  Subsequently signed graphs have turned out to be valuable in many other areas, some of which we shall allude to subsequently.

The opposite of balance is frustration.  Most signatures of a graph are unbalanced; but they can be made balanced by deleting (or, equivalently, negating) edges.  The smallest number of edges whose deletion makes the graph balanced is the \emph{frustration index}, a number which is implicated in certain questions of social psychology (\cite{PsL, MSB} et al.) and spin-glass physics (\cite{Toulouse, BMRU} et al.).  We find the frustration indices of all signed Petersen graphs (Theorem \ref{T:fr}).

The second basic property of signed graphs is switching equivalence.  Switching is a way of turning one signature of a graph into another, without changing circle signs.  Many, perhaps most properties of signed graphs are unaltered by switching, the frustration index being a notable example.  The first of our main theorems is that there are exactly six equivalence classes of signatures of $P$ under the combination of switching and isomorphism (Theorem \ref{T:types}).  Figure \ref{F:types} shows a representative of each switching isomorphism class.  In each representative the negative edges form a smallest set whose deletion makes the signed Petersen balanced.  Hence, we call them \emph{minimal signatures} of $P$ (see Theorem \ref{T:fr}).  Because there are only six switching isomorphism classes of signatures, the frustration index of every signature of $P$ can be found from those of the minimal signatures.

\begin{figure}[htbp]
\begin{center}
\includegraphics[scale=.5]{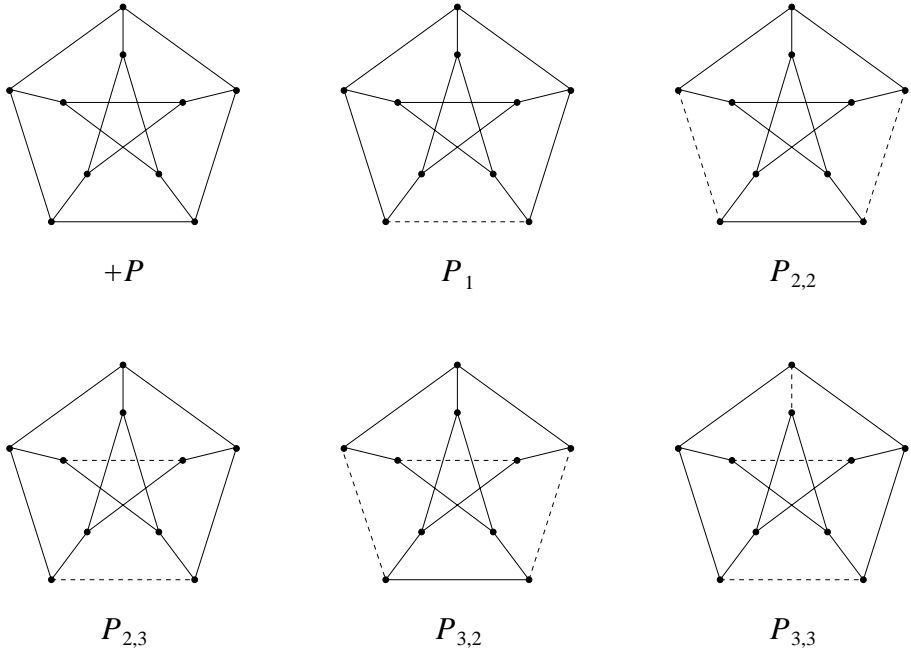}
\caption{The six switching isomorphism types of signed Petersen graph.  Solid lines are positive; dashed lines are negative.}
\label{F:types}
\end{center}
\end{figure}

The second main theorem, which occupies the bulk of this paper, is a computation of the automorphism and switching automorphism groups of the six minimal signatures (Theorem \ref{T:aut}).  An automorphism has the obvious definition: it is a graph automorphism that preserves edge signs.  This group is not invariant under switching.  It is not even truly signed-graphic, for as concerns automorphisms a signed graph is merely an edge 2-colored graph.  The proper question for signed graphs regards the combination of switching with an automorphism of the underlying graph.  
The group of switching automorphisms of a signed graph is, by its definition, invariant under switching, so just six groups are needed to know them all.  Some of the groups are trivial, but one is so complicated that it takes pages to describe it thoroughly.

Isomorphic minimal signatures may not be equivalent under the action of the switching group.  The number of switching inequivalent signatures of a given minimal isomorphism type is deducible from the order of the switching automorphism group (Section \ref{orbit}).

Two further properties are treated more concisely.  First, a signed graph can be colored by signed colors.  That leads to two chromatic numbers, depending on whether or not the intrinsically signless color 0 is accepted.  The chromatic numbers are invariant under switching (and isomorphism); thus they help to distinguish the six minimal signatures by showing their inequivalence under switching isomorphism (Theorem \ref{T:col}).  The two chromatic numbers are aspects of two chromatic polynomials, but we make no attempt to compute those polynomials, as they have degree 10.

Finally, we take a brief excursion into a natural generalization of balance called \emph{clusterability} (Section \ref{clust}).  This, like the automorphism group, is not switching invariant, but it has attracted considerable interest, most recently in connection with the organization of data (cf.\ \cite{Bansal} et al.), and has complex properties that have been but lightly explored.

Signed graphs, signed Petersens in particular, have other intriguing aspects that we do not treat.  Two are mentioned in the concluding section but they hardly exhaust the possibilities.

\section{Graphs and Signs}\label{defs}

We write $V$ and $E$ for the vertex and edge sets of a graph $\Gamma$ or signed graph $\Sigma$, except when they may be confused with the same sets of another graph.  
The complement of $X \subseteq V$ is $X^c := V \setm X$.  
The (open) neighborhood of a vertex $v$ is $N(v)$; the closed neighborhood is $N[v] = N(v) \cup \{v\}$.
A \emph{cut} in a graph is a set  $\del X := \{uv \in E : u \in X \text{ and } v \notin X\}$ where $X \subseteq V$.  
We call two or more substructures of a graph, such as edge sets or vertex sets, \emph{automorphic} if there are graph automorphisms under which any one is carried to any other.  

A \emph{signed graph} is a pair $\Sigma := (\Gamma,\sigma)$ where $\Gamma = (V,E)$ is a graph and $\sigma: E \to \signs$ is a \emph{signature} that labels each edge positive or negative.  
Hence, a \emph{signed Petersen graph} is $(P, \sigma)$.  Two examples are $+P := (P,+)$, where every edge is positive, and $-P := (P,-)$, where every edge is negative.  
The \emph{underlying graph} of $\Sigma$ is $\Sigma$ without the signs, denoted by $|\Sigma|$.  
We say $\Sigma$ is \emph{homogeneous} if it is all positive or all negative, and \emph{heterogeneous} otherwise; so $+P$ and $-P$ are the homogeneous signed Petersens.  
The set of positive edges of $\Sigma$ is $E^+$, that of negative edges is $E^-$; $\Sigma^+$ and $\Sigma^-$ are the corresponding (unsigned) graphs $(V,E^+)$ and $(V,E^-)$.  
The \emph{negation} of $\Sigma$ is $-\Sigma = (\Gamma,-\sigma)$, the same graph with all signs reversed.  
A compact notation for a signed Petersen graph with negative edge set $S$ is $P_S$. 

The sign of a circle (i.e., a cycle, circuit, or polygon) $C$ is $\sigma(C) :=$ the product of the signs of the edges in $C$.  
The most essential fact about a signed graph usually is not, as one might think, the edge sign function itself, but only the set $\cC^+(\Sigma)$ of circles that have positive sign.  If this set consists of all circles we call the signed graph \emph{balanced}.  Such a signed graph is equivalent to its unsigned underlying graph in most ways.  
We call $\Sigma$ \emph{antibalanced} if $-\Sigma$ is balanced.

\begin{prop}[{Harary \cite{NB}}]\label{P:balance}
$\Sigma$ is balanced if and only if $V$ can be divided into two sets so all positive edges are within a set and all negative edges are between the sets.
\end{prop}

We say `divided' rather than `partitioned' because one set may be empty.  If that is so, the signature is all positive.

\section{Petersen Structure}\label{structure}

The Petersen graph $P$ is the complement of the line graph of $K_5$: $P = \overline{L(K_5)}$.  Thus, its vertices $v_{ij}$ are in one-to-one correspondence with the ten unordered pairs from the set $\5$ and its fifteen edges are all the pairs $v_{ij}v_{kl}$ such that $\{i,j\} \cap \{k,l\} = \eset$.  (For legibility, in subscripts we often omit the $v$ of vertex names.)
We usually write $V$ and $E$ for $V(P)$ and $E(P)$ when discussing the Petersen graph as there can be no confusion with the vertex and edge sets of a general graph.  

For use later we want structural information about $P$.  

As $P$ has edge connectivity 3, the smallest cut has three edges.

The automorphism group $\Aut P$ is well known to be the symmetric group $\fS_5$ with action on $V$ induced by the permutations of the set $\5$.  Writing $\fS_T$ for the group of permutations of the set $T$, we identify $\Aut P$ with $\fS_{\5}$.  We use the same symbol for a permutation of $\5$ and the corresponding automorphism of $P$, as there is little danger of confusion.  $\Aut P$ carries any oriented path of length 3 to any other; hence it is also transitive on pairs of adjacent edges (distance 1) and on pairs of edges at distance 2.  (For these properties see, e.g., \cite[Section 4.4]{AGT}.)  Furthermore, $\Aut P$ carries any nonadjacent vertex pair to any other.

The maximum size of a set of independent vertices in $P$ is four.  Each maximum independent vertex set has the form $X_m := \big\{ v_{im} : i \in \5 \setm m \big\}$, and any three vertices in $X_m$ determine $m$.  For any $m,n \in \5$, $X_m$ and $X_n$ are automorphic.  
An independent set of three vertices is either the neighborhood of a vertex, or a subset of a maximum independent set $X_m$.  Deleting an independent vertex set leaves a connected graph except that $P \setm N(v) = C_6 \cupdot K_1$ and $P \setm X_m$ is a three-edge matching.  If $|W|=3$ and $W \subset X_m$, then $P \setm W$ is a tree consisting of three paths of length 2 with one endpoint in common.

\subsection{Hexagons}\label{hex}  

Each hexagon is $E(P \setm N[v])$ for a vertex $v$.  Thus there is a one-to-one correspondence between vertices and hexagons; we write $H_v = H_{lm}$ for the hexagon that corresponds to $v = v_{lm}$.  The stabilizer of $H_{lm}$ is $\fS_{\{l,m\}} \times \fS_{\5 \setm \{l,m\}}$.  A hexagon is determined by any two of its edges that have distance 2.  Furthermore, any two hexagons are automorphic.

\subsection{Matchings}\label{matchings}  

We need to know all automorphism types of a matching in $P$.  Let $M_k$ denote a matching of $k$ edges.

Matchings of 1 edge are obviously all automorphic.  

Let $M_{2,d}$ denote a pair of edges at distance $d=2$ or $3$.  Any 2-edge matching is an $M_{2,2}$ or $M_{2,3}$.  All $M_{2,2}$ matchings are automorphic because $\Aut P$ is transitive on paths of length 3.  All $M_{2,3}$ matchings are automorphic; for the proof see the treatment of $M_{3,3}$.

An $M_5$ can only be a cut between two pentagons, since $P \setm M_5$ is a 2-factor and $P$ is non-Hamiltonian.  All are clearly automorphic.

A matching of 4 edges leaves two vertices unmatched.  If they are adjacent, $M_4 = M_5\ \setm$ edge; all such matchings are automorphic.  If they are nonadjacent, say they are $v_{ik}$ and $v_{jk}$ in Figure \ref{F:m3types}.  Then $M_4$ consists of $a$ and one of the two $M_{3,2}$'s in $H_{lm}$.  
Call this type of matching $M_4'$.  
Interpreting $M_4'$ as one of the matchings in $H_{lm}$ together with one of the edges incident with $v_{lm}$, it is easy to see that all matchings of type $M_4'$ are automorphic.  Consequently, there are two automorphism classes of 4-edge matchings.

There are four nonautomorphic kinds of 3-edge matching $M_3$.  First we describe them; then we prove there are no other kinds.  

By $M_{3,3}$ we mean a set of three edges, each pair having the same distance $3$.  Each $M_{3,3}$ has the form 
$$M_{3(m)} := E(P \setm X_m) = \big\{v_{ij}v_{kl}: \{i,j,k,l\} = \5 \setm m \big\}.$$  
There are five such edge sets, one for each $m \in \5$; they partition $E(P)$.  Obviously, all the $M_{3(m)}$'s are automorphic.  
(An $M_{2,3}$ lies in a unique $M_{3,3}$, since the $M_{2,3}$ determines the value of $m$.  That implies there are 15 different $M_{2,3}$'s.)  
Permuting $\5 \setm m$ permutes the edges of $M_{3(m)}$; it follows that any $M_{2,3}$ is automorphic to any other.

We define $M_{3,2}$ to consist of alternate edges of a hexagon, say $H_{lm}$, which we call the \emph{principal hexagon} of the three edges.  There are two such sets for each hexagon, hence 20 $M_{3,2}$'s in all, and they are all automorphic to each other.  The notation $M_{3,2}$ reflects the fact that the edges in the matching all have distance two from each other.  
Each $M_{2,2}$ is contained in a unique hexagon, hence in a unique $M_{3,2}$; thus, there are 40 $M_{2,2}$'s.  

There is another way to form a matching of three edges at distance 2 from one another.  In a pentagon $v_{ij}v_{kl}v_{mi}v_{jk}v_{lm}v_{ij}$ take the edges $e=v_{kl}v_{im}$ and $f=v_{jl}v_{km}$ and the edge $a=v_{ij}v_{lm}$.  We call this type $M_3'$.  Another view of $M_3'$ is as $M_5\ \setm$ two edges.  All matchings of type $M_3'$ are automorphic but they are not automorphic to any $M_{3,2}$ because $e,f,a$ do not lie in a hexagon.

A fourth type of 3-edge matching, call it $M_{3,2/3}$, consists of $e$, $f$, and $b=v_{jk}v_{lm}$.  The distances of these edges are 2, except that $b$ and $f$ have distance 3.  All $M_{3,2/3}$'s are automorphic, but the distance pattern proves an $M_{3,2/3}$ is not automorphic to any other type.

\begin{lem}\label{L:m3types}
Every $3$-edge matching in $P$ is an $M_{3,3}$, an $M_{3,2}$, an $M_{3,2}'$, or an $M_{3,2/3}$.
\end{lem}

\begin{proof}
Let $M_3$ be a 3-edge matching.  If its edges are all at distance 3 from each other, then $M_3$ can only be $M_{3,3}$, as two edges at distance 3 have a unique edge at the same distance from both.

\begin{figure}[hbtp]
\begin{center}
\includegraphics[scale=.6]{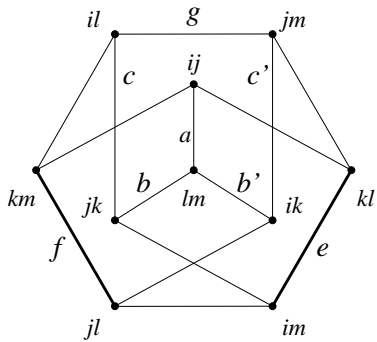}
\caption{The four kinds of $3$-edge matching in $P$.}
\label{F:m3types}
\end{center}
\end{figure}

If $M_3$ contains edges $e,f$ at distance 2, there are four potential third edges up to the symmetry that interchanges $e$ and $f$ (see Figure \ref{F:m3types}).  Choosing $g$ for $M_3$, the hexagon $H_{lm}$ contains $M_3$ so we have $M_{3,2}$.  Choosing $a$ for $M_3$, the pentagon $v_{ij}v_{kl}v_{mi}v_{jk}v_{lm}v_{ij}$ shows we have $M_3'$.  Choosing $b$, we have $M_{3,2/3}$.  Choosing $c$, we have $M_3'$ again with the pentagon $v_{im}v_{jk}v_{il}v_{km}v_{jl}v_{im}$.
\end{proof}

\section{Switching}\label{sw}

Two signed graphs, $\Sigma_1 = (\Gamma_1,\sigma_1)$ and $\Sigma_2 = (\Gamma_2,\sigma_2)$, are \emph{switching equivalent} (written $\Sigma_1 \sim \Sigma_2$) if $\Gamma_1 = \Gamma_2$ and there is a function $\zeta: V_1 \to \signs$ (a \emph{switching function}) such that $\sigma_2(vw) = \zeta(v)\sigma_1(vw)\zeta(w)$ for every edge $vw$.  We write $\sigma_2 = \sigma_1^\zeta$ and $\Sigma_2 = \Sigma_1^\zeta$; that is, we write the switched signature or graph as if we were conjugating in a group---and indeed switching is a graphical generalization of conjugation.  
Another way to state switching is to switch a vertex set $X \subseteq V$ (the connection is that $X = \zeta\inv(-)$); that means negating the sign of every edge in the cut $\del X$.  Then we write $\Sigma^X = (\Gamma,\sigma^X)$ for the switched graph.  The switching function $\zeta_X$ is defined by $\zeta_X(v) := +$ if $v \notin X$ and $-$ if $v \in X$.

Switching functions multiply pointwise:  $(\zeta\eta)(v) = \zeta(v)\eta(v)$.  Multiplication corresponds to set sum (symmetric difference) of switching sets:  $\zeta_X \zeta_Y = \zeta_{X \oplus Y}$.  
The group of switching functions is $\signs^V$.  We write $\epsilon$ for its identity element, the all-positive switching function.  Certain switching functions have no effect on $\Sigma$; that is, the action of $\signs^V$ on a signature has a kernel,
$$
\fK_\Gamma := \{\zeta : \Sigma^\zeta = \Sigma\} = \{\zeta : \zeta \text{ is constant on each component of } \Gamma\}.
$$
The kernel is independent of the signature, in fact, of everything except the partition of $V$ into vertex sets of connected components of $\Gamma$.  The quotient group is the \emph{switching group} of $\Gamma$, written 
$$\Sw\Gamma := \signs^V/\fK_\Gamma.$$
The element of this group that corresponds to a switching function $\zeta$ is $\bar\zeta$, but for simplicity of notation, we often use the same symbol $\zeta$ without the bar when it should not cause confusion.

We say $\Sigma_1$ and $\Sigma_2$ are \emph{isomorphic} (written $\Sigma_1 \cong \Sigma_2$) if there is a graph isomorphism $\psi: \Gamma_1 \to \Gamma_2$ that preserves edge signs, i.e., $\sigma_2((vw)^\psi) = \sigma_1(vw)$ for every edge.  (As we are restricting to simple graphs, $\psi$ can be treated as a bijection $V_1 \to V_2$ and $(vw)^\psi = v^\psi w^\psi$.)  
They are \emph{switching isomorphic} (written $\Sigma_1 \simeq \Sigma_2$) if $\Sigma_2$ is isomorphic to a switching of $\Sigma_1$; that is, there are a graph isomorphism $\psi: \Gamma_1 \to \Gamma_2$ and a switching function $\zeta : V_1 \to \signs$ such that $\sigma_2((vw)^\psi) = \sigma_1^\zeta(vw)$ for every edge.  

\begin{lem}[\rm\cite{Soz, CSG}] \label{L:switching}
Switching preserves circle signs.  Conversely, if two signatures of $\Gamma$ have the same circle signs, then one is a switching of the other.
\end{lem}

For instance, $\Sigma$ is balanced if and only if it is switching equivalent to the all-positive signature.  Because of this lemma, switching-equivalent signed graphs are in most ways the same.  

Lemma \ref{L:switching} shows that switching isomorphism is a true isomorphism: not of graphs or signed graphs, but of the structure on signed graphs consisting of the underlying graph and the class of positive circles, i.e., of the pair $(|\Sigma|,\cC^+(\Sigma))$ (which constitutes a type of `biased graph' \cite{BG1}).

Switching equivalence and switching isomorphism are equivalence relations on signed graphs.
An equivalence class under switching equivalence is a \emph{switching equivalence class} of signed graphs.  An equivalence class under switching isomorphism is a \emph{switching isomorphism class}.  (Many writers say `switching equivalence' when they mean `switching isomorphism', but I find it better to separate the two concepts.)

\section{Switching Isomorphism Types}\label{swisom}

The most patently obvious signatures of the Petersen graph are $+P$ and $-P$.  Two more are $P_1$, which has only one negative edge, and its negative $-P_1$, with only one positive edge.  Two more signatures are $P_{2,d}$ where $d=2,3$, which have two negative edges at distance $d$; and the last two that mainly concern us are $P_{3,d}$ for $d=2,3$, which have three negative edges, all at distance $d$; in $P_{3,2}$ the negative edges must be alternate edges of a hexagon.  In terms of our classification of matchings, $P_{k,d} := P_{M_{k,d}}$, that is, $E^-(P_{k,d}) = M_{k,d}$.  These signed graphs are illustrated in Figure \ref{F:types}.  

\begin{thm}\label{T:types}
There are exactly six signed Petersen graphs up to switching isomorphism.  They are $+P \simeq -P_{3,3}$, $P_1 \simeq -P_{2,3}$, $P_{2,2} \simeq -P_{2,2}$, $P_{2,3} \simeq -P_1$, $P_{3,2} \simeq -P_{3,2}$, and $P_{3,3} \simeq -P$.
\end{thm}

\begin{proof}
The first step is to establish the switching equivalences stated in the theorem.  
To switch $-P$ to $P_{3,3}$, switch an independent set $X=X_m$ of four vertices; this negates $\del X$ leaving three negative edges, which have distance 3.  
If we begin with $-P_1$ with positive edge $uv$, by choosing $X$ to contain neither $u$ nor $v$ we get $uv \notin \del X$ so, after switching, $uv$ retains its sign; therefore $(-P_1)^X = P_{2,3}$.  
To switch $-P_{3,2}$, where the positive edges belong to a hexagon $H_v$, switch $N[v]$.  That negates all edges except those of $H_v$, giving $P_{3,2}$ whose negative edges are the originally negative edges of the hexagon.  
To switch $-P_{2,2}$, note that the two positive edges $e$ and $f$, having distance 2, lie in a unique pentagon $J$.  Switch the three vertices of $J$ that are not incident to $e$ and the two vertices outside $J$ that are adjacent to $e$.  The result is $P_{2,2}$.

For the rest of the theorem we need two more steps.  
First, we must prove that every signed Petersen graph belongs to the switching isomorphism class of one of the six types $+P,\, P_1,\, P_{k,d}$ listed in the theorem.  That is implied by Theorem \ref{T:fr}.  
Second, we must show that none of the six types is switching isomorphic to any other.  The second step follows from the calculation of \emph{invariants} of the six switching isomorphism classes, by which we mean numbers or other objects that are the same for every element of a switching isomorphism class.  Relevant invariants are the numbers $c_5^-$ and $c_6^-$ of negative circles of lengths 5 and 6 (Theorem \ref{T:circ}), the frustration index $l$ (Theorem \ref{T:fr}, and the switching automorphism groups (Theorem \ref{T:aut}).  The six classes must be distinct because no two have all the same invariants.  
In fact, any two of $c_5^-$, $c_6^-$, and $l$ suffice to distinguish them; and the switching automorphism groups, though more difficult to find, suffice by themselves.
\end{proof}

\section{Circle Signs}\label{circ}

Lemma \ref{L:switching} leads to an effective method of distinguishing switching isomorphism classes, by comparing the numbers of negative circles of each length.  

\begin{thm}\label{T:circ}
The numbers of negative pentagons and hexagons in each of the six signed Petersen graphs of Theorem \ref{T:types} are those listed in Table \ref{Tb:circ}.
\end{thm}

\begin{table}[htb]
\begin{center}
\begin{tabular}{|c||c|c|c|c|c|c|}
\hline
$(P,\sigma)$ \vstrut{15pt}&\hbox to 3em{\hfill$+P$\hfill} &\hbox to 3em{\hfill $P_1$\hfill} &\hbox to 3em{\hfill $P_{2,2}$\hfill} &{$P_{2,3} \simeq -P_1$} &\hbox to 3em{\hfill $P_{3,2}$\hfill} &{$P_{3,3} \simeq -P$}	\\[3pt]
\hline
\vstrut{15pt}Negative $C_5$'s	&0	&4	&6	&8	&6	&12	\\
\vstrut{15pt}Negative $C_6$'s	&0	&4	&6	&4	&10	&0	\\[2pt]
\hline
\end{tabular}
\end{center}
\bigskip
\caption{The numbers of negative pentagons and hexagons in each switching isomorphism type.}
\label{Tb:circ}
\end{table}

\begin{proof}
The Petersen graph has $c_5=12$ pentagons and $c_6=10$ hexagons.  The number of cases to consider is lessened if we notice that negating $(P,\sigma)$ leaves the number $c_6^-(P,\sigma)$ of negative hexagons the same but complements the number $c_5^-(P,\sigma)$ of negative pentagons to $c_5^-(P,-\sigma) = 12-c_5^-(P,\sigma)$.

For $+P$ both numbers are 0, and the values for $-P$ follow.

In $P_1$ there are as many negative pentagons, or hexagons, as the number of each that lie on a fixed edge $e$.  There are four ways to add an edge at each end of $e$ to get a path of length 3, and each such path completes uniquely to a pentagon or hexagon.  Thus, $c_5^-(P_1)=c_6^-(P_1) =4$.  The numbers for $-P_1$ are immediate.

If we now take an edge $f$ at distance 2 from $e$, the number of negative $k$-gons equals $2(c_k^-(P_1) - d_k)$ where $d_k$ is the number of $k$-gons that contain both $e$ and $f$.  It is easy to see that $d_5=d_6=1$.  (Use the 3-path transitivity of $P$, by which under the symmetries of $P$ there is only one orbit of pairs of edges at distance 2.)  It follows that $c_5^-(P_{2,2})=c_6^-(P_{2,2})=6$.

For an $f$ at distance 3 from $e$ there is a similar calculation.  However, $f$ cannot lie in a common pentagon with $e$, so now $d_5=0$.  The value of $d_6$ is not quite obvious.  There are four ways to form a path of length 3 by extending $e$ at each end.  Inspection reveals that two of these paths cannot be completed to a hexagon on $f$, but the other two can be completed uniquely.  Thus, $d_6=2$.  We conclude that $c_5^-(P_{2,3})=8$ and $c_6^-(P_{2,3})=4$.
\end{proof}

\section{Frustration}\label{fr}

The proofs of Theorems \ref{T:types} and \ref{P:chromatic} make use of the measurement of imbalance by edges or vertices.  The \emph{frustration index} $l(\Sigma) :=$ the smallest number of edges whose deletion leaves a balanced signed graph.  It is equivalent to finding the largest number of edges in a balanced subgraph of $\Sigma$, which is the signed-graph equivalent of the maximum cut problem in an unsigned graph; in fact, $l(-\Gamma) =$ the smallest number of edges whose complement is bipartite.  The \emph{frustration number} (or \emph{vertex frustration number}) $l_0(\Sigma)$ is the smallest number of vertices whose deletion leaves a balanced signed graph.  Its complement, $|V| - l_0$, is the largest order of a balanced subgraph.  For an all-negative graph, $l_0(-\Gamma)$ is the smallest number of vertices whose deletion leaves a bipartite graph.

\subsection{Frustration index}\label{frindex}

The frustration index is the most significant way to measure how unbalanced a signed graph is.  For instance, in social psychology $l(\Sigma)$ is the minimum number of relations that must change to achieve balance.  In the non-ferromagnetic Ising model of spin glass theory the frustration index determines the ground state energy of the spin glass.
(Frustration index was called `complexity' by Abelson and Rosenberg \cite{PsL}, who introduced the idea, and `line index of balance' by Harary; my name for it was inspired by the picturesque terminology of Toulouse \cite{Toulouse}.)

Harary \cite{MSB} proved that $l(\Sigma) =$ the smallest number of edges whose negation or deletion makes the signed graph balanced.  (Negating an edge is equivalent to deleting it, so one can delete or negate the edges in any combination.)  
An edge set whose deletion leaves a balanced graph is called a \emph{balancing set} (of edges); thus, $l(\Sigma) =$ the size of a minimum balancing set.

\begin{lem}[implicit in {\cite{BMRU}}]\label{L:swfr}
Switching does not change $l(\Sigma)$.  Indeed, $l(\Sigma) = \min_\zeta |E^-(\Sigma^\zeta)|$, the minimum number of negative edges in a switching of $\Sigma$.
\end{lem}

That is, a signed graph has the smallest number of negative edges in its switching equivalence class if and only if $|E^-(\Sigma)| = l(\Sigma)$.  Let us call $\Sigma$ \emph{minimal} if it satisfies this equation.

By Lemma \ref{L:swfr} we can distinguish switching isomorphism classes by their having different frustration indices.  This helps to prove the six signed $P$'s are not switching isomorphic.

\begin{thm}\label{T:fr}
There are precisely the following six isomorphism types of minimal signed Petersen graph: $+P$, $P_1$, $P_{2,2}$, $P_{3,2}$, $P_{2,3}$, and $P_{3,3}$.  Each is the unique minimal isomorphism type in its switching isomorphism class.  
The frustration indices of the six types are as stated in Table \ref{Tb:fr}.  
\end{thm}

To find the frustration index of any signature of $P$, switch it to be minimal and consult the table.  As frustration index is an NP-complete problem (its restriction to all-negative signatures is equivalent to the well known NP-complete maximum-cut problem) that may not be so easy, but in small examples like the Petersen graph Lemma \ref{L:cutfr} is a great help.
\begin{table}[htb]
\begin{center}
\begin{tabular}{|r||c|c|c|c|c|c|}
\hline
\vstrut{15pt}$(P,\sigma)$	&\hbox to 2em{\,$+P$} &\hbox to 2em{\;\,$P_1$} &\hbox to 2em{\,$P_{2,2}$} &\hbox to 2em{\,$P_{2,3}$} &\hbox to 2em{\,$P_{3,2}$}	&\hbox to 2em{\,$P_{3,3}$}	\\[3pt]
\hline
\vstrut{15pt}$l(P,\sigma)$	&0	&1	&2	&2	&3	&3	\\[2pt]
\hline
\end{tabular}
\end{center}
\bigskip
\caption{The frustration index of each switching isomorphism type.}
\label{Tb:fr}
\end{table}

\begin{proof}
First we show that every signature of $P$ switches to one of the six.

\begin{lem}\label{L:cutfr}
If every cut in $\Sigma$ has at least as many positive as negative edges, then $l(\Sigma) = |E^-|$.  
If some cut has more negative than positive edges, then $l(\Sigma) < |E^-|$.  
\end{lem}

\begin{proof}
If $|E^-(X,X^c)| > |E^+(X,X^c)|$, then switching $X$ reduces the number of negative edges.  If $|E^-(X,X^c)| \leq |E^+(X,X^c)|$ for every $X$, then no switching can reduce the number of negative edges; so $l(\Sigma) = |E^-|$ by Lemma \ref{L:swfr}.
\end{proof}

Lemma \ref{L:cutfr} has a pleasing effect on a cubic graph.

\begin{cor}\label{C:cubicfr}
In any minimal signature of a cubic graph the negative edges are a matching.
\end{cor}

Thus, we need only examine all the automorphism types of matchings in $P$ from Section \ref{structure}.  Let $E^- = M_k$ where $0 \leq k \leq 5$.

Matchings of 0 or 1 edge are trivial: $\Sigma$ is minimal.  When $k=2$, $E^- = M_{2,2}$ or $M_{2,3}$ so we have $P_{2,2}$ or $P_{2,3}$.

When $E^- = M_5$, switching the vertices of one of the pentagons separated by $E^-$ makes all edges positive, which is $+P$.  
When $E^- =  M_5\ \setm$ edge or $E^- = M_3' = M_5\ \setm$ 2 edges, the same switching gives $P_1$ or $P_{2,2}$, respectively.

For $E^- = M_4'$ switch $\{v_{kl},v_{ij},v_{km},v_{jl}\}$.  This also results in $P_{2,2}$.

The last case is $E^- = M_{3,2/3}$.  Here we switch $\{v_{jk},v_{jl},v_{im}\}$, getting $P_{2,3}$.

This proves that every signature is switching isomorphic to one of the six basic types.  

It remains to show that each of the six types is actually minimal.  We have shown that a signature in which $E^-$ is a matching is not minimal if it is not one of the six.  Thus, if no two of the six are switching isomorphic, each must be the unique minimal element of its switching isomorphism class.  The switching invariants $c_5^-$, $c_6^-$, and $l$ are more than enough to prove that none of the six can switch to any other.  Thus, the theorem is proved.
\end{proof}

\begin{cor}\label{C:min}
In each switching equivalence class and in each switching isomorphism class of signed Petersen graphs there is exactly one minimal isomorphism type.
\end{cor}

The corollary cannot say that there is a unique minimal signature in each switching equivalence class, because that is false.  In the switching equivalence class of $-P$ the unique minimal isomorphism type is $P_{3,3}$, but the exact choice of the three negative edges is not unique.  The number of minimal graphs in that switching equivalence class equals the number of sets of three edges all at distance 3, which is 5.

It is a remarkable fact that not just some but every switching equivalence class, and every switching isomorphism class, of signed Petersens has only one minimal signature up to isomorphism.  
It is not surprising that some switching equivalence classes have this property, but that all do is.  
By way of contrast, $K_n$ (with $n\geq4$) has some switching equivalence classes with unique minimal elements, either absolutely or only up to isomorphism, and some with multiple minimal members.  
In the class of the signature $K_n(e)$, which has exactly one negative edge $e$, clearly the only minimal signed graph is $K_n(e)$.  
In the class of $-K_n$ the minimal elements are all the signatures of $K_n$ where the positive edges form a cut of maximum size, i.e., where $V(K_n)$ is partitioned into two sets whose sizes differ by at most 1 \cite{Petersdorf}.  There are many such signatures and all are switching equivalent to $-K_n$; but they are all isomorphic.  
Now assume $n=2r+1\geq5$ and consider one more signature, where the negative edges are $e_{1i}$ for $i=2,3,\ldots,r+1$ and $e_{2,3}$.  Here $E^-$ is a connected subgraph.  This signature is minimal in its switching equivalence class by Lemma \ref{L:cutfr}.  Switching $v_1$, the negative edges are $e_{1i}$ for $i=r+2,r+3,\ldots,2r+1$ and $e_{2,3}$.  The number of negative edges is unchanged, but they now form a disconnected subgraph.  Thus, this switching equivalence class contains (at least) two minimal graphs that are not isomorphic.  
We see that for $K_n$ there are switching equivalence classes whose minimal graph is unique, those in which the minimal graph is unique only up to isomorphism, and those with nonisomorphic minimal members.  

Thus, the behavior of the Petersen signatures is not totally ordinary.  I suspect it is unusual but the truth is that no one knows whether, in regard to the uniqueness of either minimal signatures or isomorphism types of minimal signatures in either their switching equivalence or isomorphism class, most graphs resemble $K_n$ or $P$.

\subsection{Frustration number}\label{frno}

The (vertex) frustration number has been less deeply explored than the frustration index, perhaps because it seems less suitable to the social psychology model and is certainly less relevant to spin glass theory.  Besides, it appears to be less subtle in distinguishing between different signatures of a graph, because most graphs have fewer vertices than edges.  Nevertheless, we find a use for it in counting colorations in Section \ref{col}.

\begin{lem}\label{L:swfrno}
Switching does not change $l_0(\Sigma)$.  Moreover, $l_0(\Sigma) \leq l(\Sigma)$ in every signed graph.
\end{lem}

\begin{proof}
The first part is obvious from Lemma \ref{L:switching} because imbalance depends only on the set of negative circles.  The second part follows from the fact that, if we delete one endpoint from each edge of a minimum balancing edge set, we get a balanced subgraph by deleting at most $l$ vertices.
\end{proof}

\begin{thm}\label{T:frno}
The frustration numbers of signed Petersen graphs are given in Table \ref{Tb:frno}.  All have frustration number equal to the frustration index.
\end{thm}

\begin{table}[htb]
\begin{center}
\begin{tabular}{|r||c|c|c|c|c|c|}
\hline
\vstrut{15pt}$(P,\sigma)$	&\hbox to 2em{\,$+P$} &\hbox to 2em{\;\,$P_1$} &\hbox to 2em{\,$P_{2,2}$} &\hbox to 2em{\,$P_{2,3}$} &\hbox to 2em{\,$P_{3,2}$}	&\hbox to 2em{\,$P_{3,3}$}	\\[3pt]
\hline
\vstrut{15pt}$l_0(P,\sigma)$	&0	&1	&2	&2	&3	&3	\\[2pt]
\hline
\end{tabular}
\end{center}
\bigskip
\caption{The frustration number of each switching isomorphism type.}
\label{Tb:frno}
\end{table}

\begin{proof}
Consult Figure \ref{F:types}.  The values for $+P$ and $P_1$ are obvious.  

A signature that has two vertex-disjoint negative pentagons cannot have $l_0<2$; if the frustration index is $2$, as in $P_{2,2}$ and $P_{2,3}$, that must be $l_0$.  

In $P_{3,2}$ the negative pentagons are the inner star and all pentagons with two outer edges.  To achieve balance we must delete an inner vertex.  Deleting one such vertex $v$ gives $P_{3,2}\setm v$, which is a subdivision of $K_4$ in which the paths corresponding to two opposite edges in $K_4$ are negative and the paths that correspond to other edges in $K_4$ are positive.  Every circle in $P_{3,2}\setm v$ that corresponds to a triangle of $K_4$ is negative.  It is impossible to make this graph balanced by deleting only one edge; hence $l_0(P_{3,2}) = 3$.

Because $P_{3,3}$ is antibalanced, every pentagon is negative.  That means a vertex set whose deletion makes for balance must cover all the pentagons.  No two vertices can do that, as one can verify by inspecting adjacent and nonadjacent pairs; but any vertex neighborhood $N(v)$ does.  Hence, $l_0(-P) = 3$.  

Comparing Tables \ref{Tb:fr} and \ref{Tb:frno} shows that $l_0 = l$ in every case.
\end{proof}

One can easily see that $l_0=l$ is not true in general.  
However, I verified that equality holds for every signature of $K_4$ or $K_{3,3}$.  I hesitantly propose:

\begin{conj}\label{Cj:cubic}
For every signed cubic graph $\Sigma$, $l_0(\Sigma) = l(\Sigma).$
\end{conj}

\section{Automorphisms and Orbits}\label{aut}

In this section we develop a general theory of switching automorphism groups of signed graphs.  Then we compute the automorphism and, more importantly, switching automorphism groups of the six basic signed Petersen graphs and their negatives.  Lastly we apply that information to find the number of isomorphic but switching-inequivalent copies of each of the six basic signatures.

We regard an automorphism of $\Gamma$ as a permutation of $V$ and we write actions as superscripts, so products are read from left to right.

\subsection{Automorphisms and switching automorphisms of signed graphs}\label{genaut}

An automorphism of a signed graph is an isomorphism with itself; that is, it is an automorphism of the underlying graph that preserves edge signs.  A \emph{switching automorphism} of a signed graph is a switching isomorphism with itself.  (As with switching isomorphisms, cf.\ near Lemma \ref{L:switching}, switching automorphisms really are automorphisms: of the biased graph $(|\Sigma|,\cC^+(\Sigma))$.)  The group of automorphisms is $\Aut(\Sigma)$ and that of switching automorphisms is $\SwAut(\Sigma)$.  

\subsubsection{Automorphisms}\label{aut}

As concerns automorphisms, a signed graph is just a graph whose edges are colored with two colors; an automorphism is a color-preserving graph automorphism.  There is not much to say except the following:

\begin{prop}\label{P:autgroup}
For a signed graph $\Sigma = (\Gamma,\sigma)$, 
$$\Aut\Sigma = \Aut\Gamma \cap \Aut\Sigma^+ = \Aut\Gamma \cap \Aut\Sigma^- = \Aut\Sigma^+ \cap \Aut\Sigma^-.$$
\end{prop}

\subsubsection{Switching permutations and switching automorphisms}\label{swaut}

Switching automorphisms are more complicated; to treat them we need precise definitions and notation.  
We begin with the action of automorphisms of $\Gamma$ upon signatures:
$$
\sigma^\alpha(v^\alpha w^\alpha) := \sigma(vw), 
$$
and $\Sigma^\alpha := (\Gamma, \sigma^\alpha).$  The action of an automorphism on a switching function is similar:
$$
\zeta^\alpha(v^\alpha) := \zeta(v).
$$
This leads to the commutation law
\begin{equation}
\zeta \alpha = \alpha \zeta^\alpha,
\label{E:commutation}
\end{equation}
because 
$$\sigma^{\zeta \alpha}(v^\alpha w^\alpha) = (\sigma^\zeta)^\alpha(v^\alpha w^\alpha) = \sigma^\zeta(vw) = \zeta(v) \sigma(vw) \zeta(w)$$ 
while 
$$\sigma^{\alpha \zeta^\alpha}(v^\alpha w^\alpha) = (\sigma^\alpha)^{\zeta^\alpha}(v^\alpha w^\alpha) = {\zeta^\alpha}(v^\alpha) \sigma^\alpha(v^\alpha w^\alpha) {\zeta^\alpha}(w^\alpha) = \zeta(v) \sigma(vw) \zeta(w).$$
Rewriting \eqref{E:commutation} as $\alpha\inv \zeta \alpha = \zeta^\alpha$, we see that the action of $\alpha$ is that of conjugation, as the notation suggests.  Rewriting it in terms of $\zeta_X$ we obtain the important equation
\begin{equation}
(\zeta_X)^\alpha = \zeta_{X^\alpha},
\label{E:commutationset}
\end{equation}
since $\zeta_X^\alpha(v^\alpha) = \zeta_X(v) = \zeta_{X^\alpha}(v^\alpha).$

Now we can define a preliminary group to the switching automorphism group.  The ground set is $\signs^V \times \Aut\Gamma$, whose elements we call, for lack of a better name, \emph{switching permutations of\/ $\Gamma$}, because when they act on a signature of $\Gamma$ they switch signs and permute the vertices.  A \emph{switching permutation of\/ $\Sigma$} is any $\zeta\gamma \in \signs^V \times \Aut\Gamma$ such that $\Sigma^{\zeta\gamma} = \Sigma$.  
The multiplication rule is 
$$(\zeta,\alpha)(\eta,\beta) = (\zeta\eta^{\alpha\inv}, \alpha\beta).$$
Because $\signs^V$ and $\Aut\Gamma$ embed naturally into $\signs^V \times \Aut\Gamma$ as $\signs^V \times \{\id\}$ and $\{\eps\} \times \Aut\Gamma$, we regard them as subgroups of $\signs^V \times \Aut\Gamma$ and write the element $(\zeta,\alpha)$ as a product, $\zeta\alpha$.  The equation of multiplication is given by the next lemma.

\begin{lem}\label{L:multinvset}
The product of switching permutations $\zeta_X\gamma$ and $\zeta_Y\xi$, where $\zeta_X, \zeta_Y \in \signs^V$ and $\gamma, \xi \in \Aut \Gamma$, is given by 
\begin{equation}
\zeta_X\gamma \cdot \zeta_Y\xi = \zeta_X\zeta_{Y^{\gamma\inv}} \cdot \gamma\xi.
\label{E:multset}
\end{equation}
The inverse of a switching permutation is 
\begin{equation}
(\zeta_X\gamma)\inv = \zeta_{X^{\gamma}} \gamma\inv.
\label{E:invset}
\end{equation}
\end{lem}

\begin{proof}
The product formula is a restatement of the previous equations.  We verify the inversion formula with a short calculation:
$$
\zeta_{X^{\gamma}} \gamma\inv \cdot \zeta_X\gamma = \zeta_{X^{\gamma}} \zeta_{X^{\gamma}}  \cdot\gamma\inv \gamma = \zeta_{X^{\gamma} \oplus X^{\gamma}} \,\id = \eps\,\id
$$
by \eqref{E:commutationset}.
\end{proof}

The commutation laws \eqref{E:commutation} and \eqref{E:commutationset} imply that the conjugate of a switching function by an automorphism is another switching function.  Consequently, $\signs^V$ is a normal subgroup.  
That makes the group of switching permutations a semidirect product of $\signs^V$ and $\Aut\Gamma$, so we write it as $\signs^V \rtimes \Aut\Gamma$.  We write $p_A$ for the projection onto $\Aut\Gamma$.

The action of $\signs^V \rtimes \Aut\Gamma$ on signed graphs $(\Gamma,\sigma)$ has kernel $\fK_\Gamma \times \{\id\}$.  The quotient group is the \emph{switching automorphism group of\/ $\Gamma$}, $$
\SwAut\Gamma := \big( \signs^V \rtimes \Aut\Gamma \big) / \big( \fK_\Gamma \times \{\id\} \big).
$$
Since $\Sw\Gamma$ can be identified with the normal subgroup $\Sw\Gamma \times \{\id\}$, $\fK_\Gamma$ with $\fK_\Gamma \times \{\id\}$, and $\Aut\Gamma$ with the subgroup $\{\bareps\} \times \Aut\Gamma$, the switching automorphism group of $\Gamma$ is a semidirect product, 
$$
\SwAut\Gamma = \Sw\Gamma \rtimes \Aut\Gamma,
$$
which projects onto $\Aut\Gamma$ by a mapping $\barp_A$.  We refer to elements of $\SwAut\Gamma$ as \emph{switching automorphisms of\/ $\Gamma$}.  (That is a slight abuse of terminology since they do not actually switch $\Gamma$; they switch signatures of $\Gamma$.)  

A switching automorphism of $\Gamma$ can be written in several equivalent ways.  As a member of $\big( \signs^V \rtimes \Aut\Gamma \big) / \big( \fK_\Gamma \times \{\id\} \big)$ it is $\overline{(\zeta,\alpha)} = \overline{\zeta\alpha}$.  As a member of $\Sw\Gamma \rtimes \Aut\Gamma$ it  is $(\bar\zeta,\alpha) = \bar\zeta\alpha$.  By the natural embeddings $\overline{\zeta\,\id} = \bar\zeta\,\id = \bar\zeta$ and $\overline{\eps\alpha} = \bar\eps\alpha = \alpha$.  In particular, the identity element of $\SwAut\Gamma$ is $\overline{\eps\,\id} = \bareps\,\id = \id$.  
Lemma \ref{L:multinvset} applies in $\SwAut\Gamma$ simply by putting a bar over the switching functions.
(Sometimes we omit the bar, as it is obvious which element of $\SwAut\Gamma$ is meant by $\zeta\alpha$.)

The switching automorphism group of $\Gamma$ contains the switching automorphism group of each signed graph $\Sigma = (\Gamma,\sigma)$.  The latter group is 
$$
\SwAut\Sigma := \{ \bar\zeta\alpha : \alpha \in \Aut\Gamma \text{ such that } \alpha: \Sigma^\zeta \cong \Sigma \}.
$$
That is, $\alpha$ must be an isomorphism from the switched signed graph to the original signed graph.  This group projects into $\Aut\Gamma$ by the mapping $\barp_A|_{\SwAut\Sigma}$, which for simplicity we also write as $\barp_A$.  
We identify $\Aut\Sigma$ with the subgroup $\{\bareps\alpha \in \SwAut\Gamma: \alpha \in \Aut\Sigma\}$.  
Note that a switching permutation of $\Sigma$ is any switching permutation of $\Gamma$ such that $\bar\zeta\gamma \in \SwAut\Sigma$.

\subsubsection{Automorphisms and switching automorphisms}\label{autswaut}

Now we can state relationships amongst the automorphisms and switching automorphisms of $\Sigma$ and the automorphisms of $\Gamma$.

\begin{prop}\label{P:autgroups}
As a function from $\SwAut\Sigma$ to $\Aut\Gamma$, $\barp_A$ is a monomorphism.  
The groups satisfy $\Aut\Sigma \leq \barp_A(\SwAut\Sigma) \leq \Aut\Gamma$.
\end{prop}

\begin{proof}
It is obvious that $\barp_A$ is a homomorphism.  To prove it is injective we examine a switching function $\zeta$ such that $\zeta\,\id$ is a switching automorphism.  That means $\Sigma^\zeta = \Sigma$, in other words, $\zeta \in \fK_\Gamma$.  But that means the only element of the form $\bar\zeta\,\id$ in $\SwAut\Sigma$ is the trivial one, $\bareps\,\id$.  Hence, $\barp_A$ is injective.

The relationships of the groups are now obvious.
\end{proof}

Another relationship makes an obvious but valuable lemma.

\begin{lem}\label{L:stab+-}
The automorphisms of $\Sigma$ are the automorphisms of $|\Sigma|$ that stabilize $\Sigma^+$, or equivalently $\Sigma^-$. 
\end{lem}

Switching automorphisms of homogeneously signed graphs are not very interesting in themselves.  

\begin{prop}\label{P:homoaut}
The automorphisms and the switching automorphisms of a homogeneous signature, $+\Gamma$ or $-\Gamma$, are the automorphisms of the underlying graph.  
\end{prop}

\begin{proof}
This follows at once from Lemma \ref{L:stab+-}.
\end{proof}

A heterogeneously signed graph, to the contrary, is likely to have switching automorphisms that are not automorphisms of the signed graph.  We see this in most, though not all, of the heterogeneous signatures of $P$.

Switching can change the automorphism group drastically.  Fortunately, the isomorphism type of the switching automorphism group is invariant under switching.  In addition, negations need not be considered separately.

\begin{prop}\label{P:negaut}
$\Aut(-\Sigma) = \Aut(\Sigma)$ and $\SwAut(-\Sigma) = \SwAut(\Sigma)$.  Also, $\SwAut(\Sigma^\zeta) \cong \SwAut(\Sigma)$ by the mapping $\bar\eta\gamma \mapsto \bar\zeta\bar\eta\gamma$.  
\end{prop}

\begin{proof}
The first statement is immediate from Lemma \ref{L:stab+-}.  

The second follows from considering how a switching automorphism acts.  $(\zeta,\alpha)$ is a switching automorphism of $\Sigma$ if and only if $\Sigma^\zeta \cong \Sigma$, the isomorphism being via $\alpha$.  This means that the same graph automorphism is an automorphism both of $(\Sigma^\zeta)^+ \cong \Sigma^+$ and of $(\Sigma^\zeta)^- \cong \Sigma^-$.  It follows that $\zeta\alpha$ is a switching automorphism of $-\Sigma$ under exactly the same conditions as it is a switching automorphism of $\Sigma$.

For the third statement we simply write down the action of $\bar\eta\alpha$:  it converts $\Sigma^\zeta$ to $(\Sigma^{\zeta})^{\eta\alpha} = \Sigma^{\bar\zeta\bar\eta\alpha}$.
\end{proof}

\begin{cor}\label{C:autpart}
Switching $\Sigma$ does not change the automorphisms in the switching automorphism group: $p_A(\SwAut\Sigma^\zeta) = p_A(\SwAut\Sigma)$ for any switching function $\zeta$.
\end{cor}

\begin{proof}
Examine the mapping in Proposition \ref{P:negaut}.
\end{proof}

Suppose $\zeta\alpha$ is a switching automorphism.  Since $(\Sigma^\zeta)^- \cong \Sigma^-$, the switching cannot change the number of negative edges.  As switching means negating the signs of edges in a cut, the cut must have equally many positive and negative edges.  Thus we have a necessary condition for a switching automorphism:

\begin{prop}\label{P:swautcut}
If $\zeta_X\alpha$ is a switching automorphism of $\Sigma$, then $\del X$ has equally many edges of each sign.
\hfill\qed
\end{prop}

\subsubsection{Coset representation}\label{cosetrep}

We treat multiplication in a switching automorphism group $\SwAut\Sigma$ through the left cosets of $\Aut\Sigma$.  Choose a system $\bar R$ of representatives of the cosets and a system $R$ of representatives $\zeta_X\gamma_X \in \signs^V \rtimes \Aut\Gamma$ of the elements $\bar\zeta_X\gamma_X \in \bar R$.  
Then $\SwAut\Sigma$ is the disjoint union of the left $\bar R$-cosets of $\Aut\Sigma$:
\begin{equation}
\SwAut\Sigma = \bigcup_{\zeta_X\gamma_X \in R} \bar\zeta_X\gamma_X \Aut\Sigma .
\label{E:swautrep}
\end{equation}
Thus we have two levels of representation:  a switching automorphism $\bar\zeta_X\gamma_X$ representing each coset, and a switching permutation $\zeta_X\gamma_X$ to represent each $\bar\zeta_X\gamma_X \in \bar R$.  
Note that $\zeta_X$ and $\zeta_{X^c} = -\zeta_{X}$ are equally valid representatives of $\bar\zeta_X$; thus we can choose $X$ so that $|X|\leq\frac12|V|$.  

\begin{prop}\label{P:cosetsw}
The following three statements about two switching automorphisms, $\bar\zeta_X\gamma$ and $\bar\zeta_Y\xi \in \SwAut\Sigma$, are equivalent.
\begin{enumerate}[{\rm(i)}]
\item They belong to the same coset of $\Aut\Sigma$ in $\SwAut\Sigma$.
\item They have the same switching operation, $\bar\zeta_X = \bar\zeta_Y$.
\item $\gamma$ and $\xi$ belong to the same coset of $\Aut\Sigma$ in $\Aut\Gamma$.
\end{enumerate}
\end{prop}

\begin{proof}
The switching automorphisms are in the same coset $\iff$ there is an $\alpha \in \Aut\Sigma$ such that $\bar\zeta_X\gamma = \bar\zeta_Y\xi \alpha$.  Because $\SwAut\Sigma \subseteq \SwAut\Gamma$ and $\bar p_A$ is a monomorphism, this implies (iii) $\gamma = \xi\alpha \in \xi \Aut\Sigma$ and (ii) $\bar\zeta_X = \bar\zeta_Y$.

Now suppose (ii), i.e., there are cosets $\bar\zeta_X\gamma\Aut\Sigma$ and $\bar\zeta_X\xi\Aut\Sigma$ with the same switched set $X$.  Then $(\bar\zeta_X\gamma)\inv(\bar\zeta_X\xi) \in \Aut\Sigma$.  Simplifying, $(\bar\zeta_X\gamma)\inv(\bar\zeta_X\xi) = \gamma\inv \bar\zeta_X\inv \bar\zeta_X \xi = \gamma\inv \xi$.  Thus, $\gamma\inv \xi \in \Aut\Sigma$, which implies (iii).

Finally, suppose (iii), i.e., $\xi = \gamma\alpha$.  Then $\bar\zeta_X\gamma = \bar\zeta_Y\gamma\alpha$.   As in the first part of the proof, this implies (ii) $\bar\zeta_X = \bar\zeta_Y$ and consequently $\bar\zeta_Y\xi = \bar\zeta_X\gamma\alpha \in \bar\zeta_X\gamma \Aut\Sigma$, which is (i).
\end{proof}

\begin{cor}\label{C:cosetrepsw}
Each left coset representative $\bar\zeta_X\gamma_X \in \bar R$ has a different switching function $\bar\zeta_X$.
\end{cor}

By Corollary \ref{C:cosetrepsw}, $X$ determines $\gamma_X$; thus, we define 
$$\rho_X := \zeta_X\gamma_X := \text{the unique element of } R \text{ that has switching set } X.$$  
Also, define $\rho_{X^c} = \zeta_{X^c}\gamma_{X^c}$.  Then $\bar\rho_X = \bar\rho_{X^c}$ because $\bar\zeta_X = -\bar\zeta_X = \bar\zeta_{X^c}$.  
Thus, assuming $\Sigma$ is connected each $\bar\rho \in \bar R$ has two associated switching sets, $X$ and $X^c$, each of which serves equally well to represent $\bar\rho$.  (There are only two because $\fK_\Gamma = \{\pm\eps\}$.)  

The task now is to express the product of switching automorphisms in terms of coset representatives.  In the next subsection we do that for the more complicated signed Petersen examples by setting up multiplication tables for $R$, which combine with a general formula to give all products.  Here we explain the format of such tables and obtain the general product formula.

The product of representatives, $\zeta_X\gamma_X \cdot \zeta_Y\xi$, has the form $(\zeta_U\gamma_U)\nu$ where $\zeta_U\gamma_U \in R$ and the permutation $\nu \in \Aut\Sigma$ is a correction due to the fact that the product of representatives need not be a representative itself.  We need formulas for $U$ and $\nu$ in terms of $R$.  (The application to $\bar R$ consists merely of placing bars over the switching functions.)
For simplicity we assume $\Sigma$ is connected, to ensure that $\bar\zeta_U$ is represented only by $\zeta_U$ or $\zeta_{U^c} = -\zeta_U$.

\begin{prop}\label{P:genmult}
Assume $\Sigma = (\Gamma,\sigma)$ is connected.  For switching automorphisms $(\bar\zeta_X\gamma_X)\alpha$ and $(\bar\zeta_Y\gamma_Y)\beta$, where $\zeta_X\gamma_X, \zeta_Y\gamma_Y \in R$ and $\alpha, \beta \in \Aut\Sigma$, there is the multiplication formula 
\begin{equation}
(\zeta_X\gamma_X)\alpha \cdot (\zeta_Y\gamma_Y)\beta = (\pm\zeta_U\gamma_U) \nu \cdot \alpha\beta,
\label{E:genmult}
\end{equation}
where $U = X \oplus Y^{\alpha\inv\gamma_X\inv}$, $\gamma_U$ and the sign are determined by $\pm\zeta_U\gamma_U \in R$, and $\nu = \gamma_U\inv \gamma_X\gamma_Y^{\alpha\inv} \in \Aut\Sigma$.
\end{prop}

\begin{proof}
Most of the proof is a calculation: 
\begin{align*}
(\zeta_{X} \gamma_X)\alpha \cdot (\zeta_Y \gamma_Y)\beta 
&= (\zeta_{X} \gamma_X) (\zeta_Y^{\alpha\inv}\gamma_Y^{\alpha\inv}) \cdot \alpha\beta \\
&= (\zeta_{X} \gamma_X) (\zeta_{Y^{\alpha\inv}}\gamma_Y^{\alpha\inv}) \cdot \alpha\beta \\
&= (\zeta_{X}\zeta_{Y^{\alpha\inv}}^{\gamma_X\inv}) (\gamma_X\gamma_Y^{\alpha\inv}) \cdot \alpha\beta \\
&= (\zeta_{X}\zeta_{Y^{\alpha\inv\gamma_X\inv}}) (\gamma_X\gamma_Y^{\alpha\inv}) \cdot \alpha\beta \\
&= \zeta_{X \oplus Y^{\alpha\inv\gamma_X\inv}} (\gamma_X\gamma_Y^{\alpha\inv}) \cdot \alpha\beta. 
\end{align*}
By Corollary \ref{C:cosetrepsw}, $\bar\zeta_U$ determines $\gamma_U \in \Aut\Gamma$ such that $\bar\zeta_U\gamma_U \in \bar R$; consequently, 
\begin{align*}
(\zeta_{X} \gamma_X)\alpha \cdot (\zeta_Y \gamma_Y)\beta 
&= (\pm\zeta_{U}\gamma_U) (\gamma_U\inv \gamma_X\gamma_Y^{\alpha\inv}) \cdot \alpha\beta .
\end{align*}
The sign is determined by whether $U := X \oplus Y^{\alpha\inv\gamma_X\inv}$ or its complement is the set $U'$ switched by the representative $\zeta_{U'}\gamma_X \in R$.  In the former case $U' = U$ and the sign is $+$, while in the latter case $U' = U^c$, which introduces the minus sign.  $U'$ must be one or the other because switching any other set will give some edge in a spanning tree a different sign.

The reason $\nu \in \Aut\Sigma$ is that, by the definition of $U$, $(\bar\zeta_{X} \gamma_X)\alpha \cdot (\bar\zeta_Y \gamma_Y)\beta \in \bar\zeta_U\gamma_U \Aut\Sigma$.  Thus, $\nu \cdot \alpha\beta \in \Aut\Sigma$, which entails that $\nu \in \Aut\Sigma$.
\end{proof}

Ideally, to use Equation \eqref{E:genmult} in conjunction with the multiplication table of $R$, one first finds $Y' := Y^{\alpha\inv}$, then looks up the product $(\pm\rho_U) \nu = \rho_X \rho_{Y'}$ in the table and combines with $\alpha\beta$.  (It is not necessary to find $Y^{\alpha\inv\gamma_X\inv}$ or $U$.)  
For this method to work, $R$ should be closed under conjugation by $\Aut\Sigma$.  
With $\Sigma = P_{3,2}$ and $P_{3,3}$ one can choose $R$ suitably; that is, so it is a union of orbits of $\Aut\Sigma$ acting on $\SwAut\Sigma$.  However, it may not always be possible to choose such an ideal system of representatives.

\begin{question}\label{Q:systemofreps}
Does a system of representatives $\bar R$ that is closed under conjugation by $\Aut\Sigma$ exist for every signed graph?  
\end{question}

A necessary condition for such a system is that, if $(\zeta_X\gamma_X)^\alpha$ is in the same coset as $\zeta_X\gamma_X$, then it must equal $\zeta_X\gamma_X$.  Thus $\gamma_X$ should commute with every automorphism $\alpha$ of $\Sigma$ for which $\bar\zeta_{X^\alpha} = \bar\zeta_X$ (equivalently when $\Sigma$ is connected, $X^\alpha = X$ or $X^c$).

\subsection{Petersen automorphisms and switching automorphisms}\label{petaut}

Here we find the automorphism and switching automorphism groups of the six minimal signed Petersen graphs and their negations.  Between them they have six automorphism groups and six switching automorphism groups, but only four abstract types of switching automorphism group.  
By Proposition \ref{P:negaut} the negative signature, $(P,-\sigma)$, has exactly the same groups as does $(P,\sigma)$, and furthermore $\SwAut(P_{3,3}) \cong \SwAut(-P)$ and $\SwAut(P_{2,3}) \cong \SwAut(-P_1)$.  By Proposition \ref{P:homoaut} and \ref{P:homoaut}, both groups of $+P$ and $-P$ equal $\Aut(P) = \fS_5$.  
Thus, as abstract groups we have five automorphism groups and three switching automorphism groups to discover; but there are five switching automorphism groups to find as explicit subgroups of $\SwAut P$.

\begin{thm}\label{T:aut}
The abstract automorphism and switching automorphism groups of the minimal signed Petersen graphs and their negatives are as shown in Table \ref{Tb:aut}.  As subgroups of $\SwAut P$ they are shown in Table \ref{Tb:autexact}.
\end{thm}

\begin{table}[htb]
\begin{center}
\begin{tabular}{|c||c|c|}
\hline
\vstrut{15pt}$(P,\sigma)$	&$\Aut(P,\sigma)$	&$\SwAut(P,\sigma)$	
\\[3pt]	
\hline
\vstrut{15pt}$+P$, $-P$	&$\fS_5$	&$\fS_5$	
\\[3pt]
\vstrut{15pt}$P_1$, $-P_1$	&$\fD_4$	&$\fD_4$	
\\[3pt]
\vstrut{15pt}$P_{2,2}$, $-P_{2,2}$	&$\fZ_2$	&$\fV_4$	
\\[6pt]
\vstrut{15pt}$P_{2,3}$, $-P_{2,3}$	&$\fD_4$	&$\fD_4$	
\\[3pt]
\vstrut{15pt}$P_{3,2}$, $-P_{3,2}$	&$\fS_3$	&$\fA_5$	
\\[3pt]
\vstrut{15pt}$P_{3,3}$, $-P_{3,3}$	&$\fS_4$		&$\fS_5$	
\\[3pt]
\hline
\end{tabular}
\end{center}
\bigskip
\caption{The automorphism and switching automorphism groups of the minimal signed Petersens and their negatives.  $\fS_k$, $\fA_k$, $\fD_k$, and $\fZ_k$ are the symmetric and alternating groups on $k$ letters, the dihedral group of a $k$-gon, and the cyclic group of order $k$.  $\fV_4$ is the Klein four-group.}
\label{Tb:aut}
\end{table}
\begin{table}[htb]
\begin{center}
\begin{tabular}{|l||c|c|c|c|}
\hline
\vstrut{15pt}$(P,\sigma)$	&$\Aut(P,\sigma)$	&$\SwAut(P,\sigma)$	
\\[3pt]	
\hline\hline
\vstrut{16pt}$+P$, $-P$	
&$\fS_{\5}$	
&$\{\bareps\}\times\fS_{\5}$	
\\[5pt]
\hline
\vstrut{20pt}\parbox{4cm}{$P_1$, $-P_1$ with \\ $E^-=\{v_{ij}v_{kl}\}$}	
&$\big\langle(ij),(ikjl)\big\rangle$	
&$\{\bareps\}\times\big\langle(ij),(ikjl)\big\rangle$	
\\[10pt]
\hline
\vstrut{20pt}\parbox{4cm}{$P_{2,2}$, $-P_{2,2}$ with \\ $E^-=\{v_{il}v_{jm},v_{kl}v_{im}\}$}	
&$\big\langle(jk)(lm)\big\rangle$	
&$\big\langle \bareps(jk)(lm), \zeta_{\{{jm},{kl}\}}(jl)(km) \big\rangle$	
\\[10pt]
\hline
\vstrut{20pt}\parbox{4cm}{$P_{2,3}$, $-P_{2,3}$ with \\ $E^-=\{v_{ik}v_{jl},v_{il}v_{jk}\}$}	
&$\big\langle(ij),(ikjl)\big\rangle$	
&$\{\bareps\}\times\big\langle(ij),(ikjl)\big\rangle$	
\\[10pt]
\hline
\vstrut{20pt}\parbox{5.2cm}{$P_{3,2}$, $-P_{3,2}$ with \\ $E^-=\{v_{il}v_{jm},v_{kl}v_{im},v_{jl}v_{km}\}$}
&$\big(\fS_{\{i,j,k\}}\times\fS_{\{l,m\}}\big)^+$	
&See Equation \eqref {E:P32set} 
\\[10pt]
\hline
\vstrut{20pt}\parbox{4.8cm}{$P_{3,3}$, $-P_{3,3}$ with \\ $E^-=\{v_{ij}v_{kl},v_{ik}v_{jl},v_{il}v_{jk}\}$}
&$\fS_{\{i,j,k,l\}}$	
&See Equation \eqref {E:P33setgen} 
\\[10pt]
\hline
\end{tabular}
\end{center}
\bigskip
\caption{The exact groups corresponding to specific negative edge sets.  $i,j,k,l,m$ are the five elements of $\5$, in any order.  For $\fG \leq \fS_n$, $\fG^+$ denotes the set of even permutations in $\fG$.  $\zeta_X$ is the switching function that switches $X \subseteq V$ (with $ij$ denoting vertex $v_{ij}$ for readability).}
\label{Tb:autexact}
\end{table}

We preface the proof with a structural lemma.

\begin{lem}\label{L:cutswitch}
Let $(P,\sigma)$ be a minimal signature of $P$.  Suppose $\del X$ is a cut that contains equally many edges of each sign, as when $X$ is switched in a switching automorphism.  Then 
\begin{enumerate}[{\rm(a)}]
\item $|\del X|=4$, $X=V(e_0)$ for some edge $e_0$, and $(P,\sigma) = P_{k,2}$ for $k=2$ or $3$, or 
\item $|\del X|=6$, $X = V(Q)$ for a path $Q$ of order $4$, and $(P,\sigma) = P_{3,2}$, or 
\item $|\del X|=6$, $X = N[v]$ for some vertex $v$, and $(P,\sigma) = P_{3,3}$.
\end{enumerate}
\end{lem}

Note that Lemma \ref{L:cutswitch} does not apply to a switching automorphism in which there is no switching.

\begin{proof}
Suppose the subgraph $P{:}X$ induced on $X$, with edge set $E{:}X$, is disconnected; then $\del X$ is the disjoint union of two or more cuts, hence it has at least 6 edges.  As $(P,\sigma)$ is minimal, there are no more than three negative edges; hence $|\del X| = 6$ and $X$ consists of two nonadjacent vertices.  Then $\del X$ does not contain three independent edges; by Corollary \ref{C:cubicfr} this case is impossible.

Therefore $P{:}X$ is connected, so $|\del X| = 3|X| - 2|E{:}X|$.  As $|\del X|$ is even, this implies $|X|$ is even, so we may assume $|X| \leq 4$.  Then $P{:}X$ is acyclic; being connected, it is a tree.  Consequently $|E{:}X| = |X|-1$ and we deduce that $|\del X| = |X|+2$.  

If the cut has four edges, $|X| = 2$; so $X = V(e_0)$ for some edge $e_0$ and $\del X$ consists of the four edges adjacent to $e_0$.  Amongst them the largest distance is 2.  It follows that $(P,\sigma) = P_{k,2}$ as in (a).

If the cut has six edges, $|X|=4$.  $P{:}X$ is a tree which may be either a path $Q$ of length $4$ or a vertex star.  If it is a path $Q$, then $X=V(Q)$ and the six edges of $\del X$ contain no three edges at distance 3 from one another.  Hence, $d=2$ and we have (b).  If $P{:}X$ is a vertex star, $X=N[v]$ for some $v \in V$.  
In this case $d=3$, for it is not possible to choose three edges in $\del X$ whose distances are all 2.  Thus, we are in case (c).
\end{proof}

\begin{proof}[Proof of Theorem \ref{T:aut}]
In the course of the proof we establish many important facts about the groups, in particular multiplication tables for the most complicated ones, $\SwAut P_{2,3}$ and $\SwAut P_{3,3}$.  The proofs of these facts could not easily be separated from that of the main theorem so it seemed best, though unconventional, to incorporate them all including their formal statements into one large proof.  In order to keep the reader (and the author) from getting lost, the proof is divided into subsections treating different aspects.

The groups of $+P$ follow from Proposition \ref{P:homoaut}.  We take up the others in turn.  

\subsubsection{Signatures of type $P_1$.}  \label{typeP1}

The automorphism group of $P_1$ is the stabilizer of an edge in $\Aut P$.  Suppose $P_1$ to have negative edge $e=v_{ij}v_{kl}$; i.e., it is $P_{\{e\}}$.  An automorphism $\alpha$ can preserve the vertices; then it is in the four-element group generated by $(ij)$ and $(kl)$.  Or, it can exchange the vertices; this is done, for instance, by a permutation $(ikjl)$.  The group $\big\langle(ij),(kl),(ikjl)\big\rangle$ is the dihedral group of a square with corners labelled, in circular order, $i,k,j,l$; it is generated by $(ij)$ and $(ikjl)$.

Due to Proposition \ref{P:swautcut} and the fact that no cut in $P$ has fewer than three edges, there are no switching automorphisms of $P_1$ other than its automorphisms.

\subsubsection{Signatures of type $P_{2,d}$.}  \label{typeP2d}

We write $P_{\ef}$ for $P_{2,d}$ with negative edges $e$ and $f$.  An automorphism of $P_{\ef}$ preserves $\ef$.  

In $P_{2,3}$ there is a unique third edge $g$ at distance $3$ from $e$ and $f$ forming a matching $M_{3(m)}$.  As any edge in $M_{3(m)}$ determines the whole matching, an automorphism of $P$ that stabilizes $\ef$ must fix $g$, and vice versa.  Thus, $\Aut P_{\ef} = \Aut P_{\{g\}}$.  

In $P_{2,2} = P_{\ef}$, $e$ and $f$ are at distance 2 in a hexagon $H_{lm}$.  The hexagon is uniquely determined by $\ef$.  There is a unique edge $g$ at distance 2 from $e$ and $f$ in $H$.  Let $e=v_{il}v_{jm}$, $f=v_{kl}v_{im}$, and $g=v_{jl}v_{km}$.  
Since an automorphism $\alpha$ of $P_{\ef}$ preserves distance, the adjacent vertices $v_{jm}, v_{kl}$ of $E^-$ are either fixed or interchanged, and the remaining vertices $v_{il}, v_{im}$ are also fixed or interchanged.  This implies that $i$ is fixed under $\alpha$, so $\alpha$, if not the identity, transposes $l$ and $m$, and consequently $\alpha = \id$ or $(jk)(lm)$.  Hence, $\Aut P_{\ef} = \big\langle(jk)(lm)\big\rangle \cong \fZ_2$, the cyclic group of order 2.

Now let us examine possible switching automorphisms $\zeta_X\gamma$ of $P_{\ef} = P_{2,d}$ for $d=2,3$.  By Lemma \ref{L:cutswitch} $|\del X| = 4$ and $P_{\ef} = P_{2,2}$.  It follows that a nontrivial switching of $P_{2,3}$ cannot be isomorphic to $P_{2,3}$, so $\SwAut P_{2,3} = \{\bareps\}\times\Aut P_{2,3}$.  There is a nontrivial switching by $X = \{v_{jm},v_{kl}\}$ forming new negative edges $e' = v_{jm}v_{ik}$ and $f' = v_{kl}v_{im}$, so $\gamma$ must fix $i$ and transpose either $j,l$ and $k,m$ or else $j,m$ and $k,l$.  Thus, $\gamma = (jl)(km)$ or $(jm)(kl)$.  We conclude that 
$$
\SwAut P_{\ef} = \{ \bareps\,\id,\ \bareps(jk)(lm),\ \zeta_{\{{jm},{kl}\}}(jl)(km),\ \zeta_{\{{jm},{kl}\}}(jm)(kl) \}.
$$

\subsubsection{Signatures of type $P_{3,d}$.}  \label{typeP3d}

The next groups are those of $P_{3,d} = P_{\efg}$ for $d=2,3$.  For each distance $d$ choose the same negative edges $e,f,g$ as in the previous analyses of $P_{2,d}$.  
In $P_{3,2}$ the negative edges lie in the hexagon $H = H_{lm} = P \setm N[v_{lm}]$.  
In $P_{3,3}$ the negative edges are $e=v_{ij}v_{kl},\ f=v_{ik}v_{jl},\ g=v_{il}v_{jk}$, so $E^- = M_{3(m)}$.

We begin with the automorphism groups.

To determine $\Aut P_{3,2}$, note that the hexagon containing $e,f,g$ is $H_v = P \setm N[v]$ for $v=v_{lm}$.  An automorphism $\alpha$ of $P_{\efg}$ must fix $v$ and thus must fix or exchange $l$ and $m$.  It can also permute the other indices $i,j,k$.  Suppose $\alpha$ fixes $l$ and $m$.  As the vertices of $H_v$, in order, are $v_{li},v_{mk},v_{lj},v_{im},v_{lk},v_{jm}$, with vertex indices alternating between $l$ and $m$, and as $\alpha$ must preserve the set $\efg$, it must rotate $H_v$ by a multiple of one-third of a full rotation.  That means it permutes $i,j,k$ cyclically, so it is a power of $(ijk)$.  Now suppose $\alpha$ exchanges $l$ with $m$.  Then it reverses the direction of $H_v$, so in order to leave $\efg$ invariant it must fix one of $e,f,g$ and one of $i,j,k$; thus, $\alpha = (ij)(lm)$, $(ik)(lm)$, or $(jk)(lm)$.  The conclusion is that $\alpha$ is an even permutation of $\5$ and is an element of $\fS_{\{l,m\}} \times \fS_{\{i,j,k\}}$.  Thus, $\Aut P_{\efg} = (\fS_{\{l,m\}}\times\fS_{\{i,j,k\}})^+$, the superscript $+$ denoting even permutations only.  As the factor $(lm)$ is predictable by evenness given the $\fS_{\{i,j,k\}}$ part of an automorphism, $\Aut P_{3,2} \cong \fS_3$.

The automorphism group of $P_{3,3}$ is determined by the fact that the negative edge set $\efg = M_{3(m)}$.  An automorphism permutes $e,f,g$, whence it permutes $i,j,k,l$ arbitrarily and fixes $m$.  Thus, $\Aut P_{\efg} = \fS_{\{i,j,k,l\}} \cong \fS_4$.

Now we examine the switching automorphism groups.  We assume $P_{3,d} = P_{\efg}$ switches by $X$ to $P_{\{e',f',g'\}}$.  Lemma \ref{L:cutswitch} presents three cases to consider.

\emph{In Case (a)}, $d=2$ so the three negative edges lie in the hexagon $H_v$.  As switching changes two edges from negative to positive, this resembles the case of $P_{2,2}$, but now there are three possible switching sets $X$, namely $X = \{w,x\}$ for each positive edge $wx$ in $H_v$.  
\begin{figure}[hbt]
\begin{center}
\includegraphics[scale=.6]{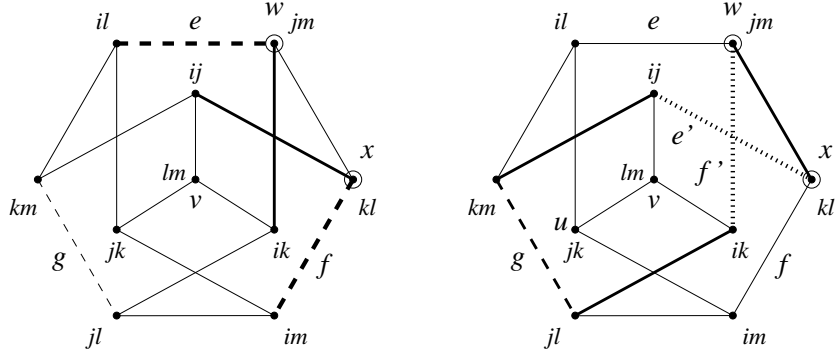}
\caption{Switching two positively adjacent vertices (circled) on the principal hexagon in $P_{3,2}$ for Case (a).  Left:  $P_{\efg}$, before switching.  The principal hexagon $H_v$ is the outer hexagon.  Heavy lines indicate the cut $\del X$.  Right:  $P_{\efg}^X$, after switching $X = \{v_{jm},v_{kl}\}$.  Heavy lines indicate the new principal hexagon $H_u$ and dotted lines mark the two new negative edges.}
\label{F:sw2}
\end{center}
\end{figure}
Switching $X$ gives a $P_{3,2}$ with negative edge set $\{e',f',g'\} \subseteq H_u$.  The vertex $u$ can be described in terms of the 3-edge path in $H_v$ centered upon $wx$:  there is a unique pentagon containing this path, and $u$ is its one vertex not in $H_v$.  It follows that each different edge $wx$ yields a different principal hexagon after switching.  
Now suppose $X = \{v_{jm},v_{kl}\}$; then $u = v_{jk}$ and $P_{\efg}^X$ is isomorphic to $P_{\efg}$ by the even permutation $\gamma_X:=(jm)(kl)$.  Similarly, each of the other two switching sets $X$ gives $P_{\efg}^X$ which is isomorphic to $P_{\efg}$ by an even permutation.  
It follows from Proposition \ref{P:cosetsw} that each different $\zeta_X\gamma_X$ belongs to a different left coset of $\Aut P_{\efg}$ in $\SwAut P_{\efg}$.  Thus we have three cosets besides $\Aut P_{\efg}$ itself.  

The three coset representatives are a single orbit of the action of $\Aut P_{\efg}$ on $\SwAut P_{\efg}$.  To prove this we may point to symmetry or we may compute the action on a coset representative $\bar\zeta_X\gamma_X$, or rather on the switching permutation $\zeta_X\gamma_X$.  
The argument from symmetry is that each switching automorphism is obtained from one of them, say $\bar\zeta_{jm,kl} (jm)(kl)$, by rotating Figure \ref{F:sw2} through $120^\circ$ once or twice.  The rotation is carried out by the permutation $(kji)$.  As for a double transposition, say $(jk)(lm) \in \Aut P_{\efg}$, applying it reflects the figure across a line parallel to $v_{jk}v_{lm}$ and therefore does not change the switching automorphism $\bar\zeta_{jm,kl} (jm)(kl)$; the other double transpositions similarly fix the other switching automorphisms.
For the computational proof, first, the action of powers of $(ijk)$:
\begin{equation}
\begin{aligned}{}
[ \zeta_{jm,kl} (jm)(kl) ] ^{(ijk)} &= \zeta_{km,il} (km)(il), \\
[ \zeta_{jm,kl} (jm)(kl) ]^{(kji)} &= \zeta_{im,jl} (im)(jl).
\label{E:P32aorbit}
\end{aligned}
\end{equation}
This shows the chosen representatives are in one orbit.  Next, the action of $(jk)(lm)$:
\begin{align*}
[\zeta_{jm,kl}(jm)(kl)]^{(jk)(lm)} &= \zeta_{jm,kl}(jm)(kl).
\end{align*}
As $\Aut P_{\efg} = \langle(ijk)\rangle \cup (jk)(lm) \langle(ijk)\rangle$, this proves there are no other switching permutations in the orbit.  The computational proof gives the slightly stronger result that the switching permutations, not only the switching automorphisms, are a whole orbit of $\Aut P_{\efg}$.

\emph{In Case (b)}, $d=2$ and $P{:}X$ is a path $wxyz$.  Again $e,f,g$ are alternating edges on $H_v$.  

Given $H_v$, we need to know which sets $X = \{w,x,y,z\}$ can be.  To determine that, we reverse the question; we fix $X$ and ask which hexagons $H_v$ can be.  (There are 60 paths of length 3, but as $\Aut P$ is transitive on them, there is only one type.)  Since $e,f,g \in \del X$, it must be true that $|H_v \cap \del X| = 3$.  One finds by checking every vertex of $P$ that only two hexagons $H_v$ have this property; the vertices $v$ are the neighbors of $x$ and $y$ in $X^c$.  By choice of notation, we may assume $v$ is adjacent to $y$.

\begin{figure}[hbt]
\begin{center}
\includegraphics[scale=.6]{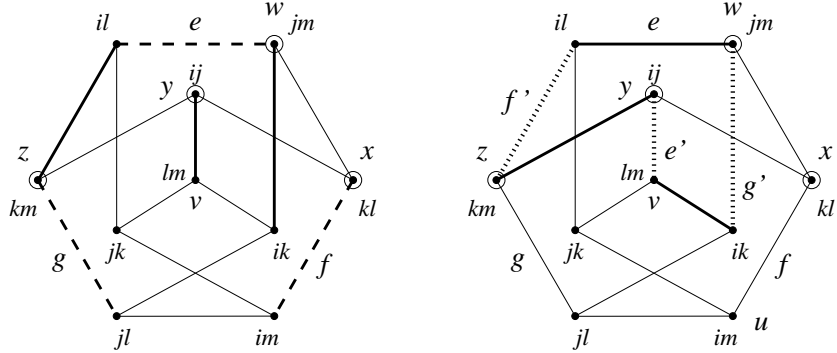}
\caption{Switching the four vertices of a path in $P_{3,2}$ for Case (b).  Left:  $P_{\efg}$, before switching $X = \{w,x,y,z\}$.  The principal hexagon $H_v$ is the outer hexagon.  Heavy lines indicate the cut $\del X$.  Right:  $P_{\efg}^X$, after switching.  Heavy lines indicate the new principal hexagon $H_u$ and dotted lines mark the new negative edges.}
\label{F:sw4path}
\end{center}
\end{figure}
Now we can describe the relationship between the path $wxyz$ and $P_{\efg}$.  The path begins with the positive edge $wx$ of $H_{v}$, which is followed by $y \notin V(H_v)$, and then ends at $z$ in $H_v$.  The original negative edges $e,f,g$ are the alternating triple in $H_v$ that excludes $wx$.  The vertex $y$ is the neighbor of $x$ along $H_v$.
Thus, there are six possible paths for $wxyz$.  Once we choose $w$ and $x$, the rest is determined.  

After switching $X=\{w,x,y,z\}$ we again have three negative edges on a hexagon; this hexagon is $H_u$ where $u$ is the neighbor of $x$ along $H_v$.  $H_v \cap H_u$ is the 2-edge path from $w$ to $z$ in $H_v$; the first edge is one of $e,f,g$ and hence positive (after switching), while the next, call it $e'$, is negative.  The negative edge set of $P_{\efg}^X$ consists of $e'$ and the edges $f', g'$ at distance 2 from it along $H_u$.  Thus, $P_{\efg}$ switches to $P_{\{e',f',g'\}}$.

To find a permutation $\alpha$ by which $P_{\efg}^X$ is isomorphic to $P_{\efg}$, we need only examine one case, because each path $wxyz$ maps to any other, $w'x'y'z'$, by the unique automorphism of $P_{\efg}$ which carries $(w,x)$ to $(w',x')$.  
Let $e=v_{il}v_{jm}$, $f=v_{kl}v_{im}$, and $g=v_{jl}v_{km}$, so $v = v_{lm}$, and let the path $wxyz = v_{jm}v_{kl}v_{ij}v_{km}$.  Then $u = v_{im}$.  The even permutation $(ilm)$ is one choice for the desired isomorphism.  The switching automorphism of $P_{\efg}$ is $\bar\zeta_{\{{jm},{kl},{ij},{mk}\}} (ilm)$.  (In the notation of Section \ref{cosetrep} this is $\bar\rho_{\{{jm},{kl},{ij},{mk}\}}$.)
 
These six switching automorphisms are another orbit of $\Aut P_{\efg}$ acting on $\SwAut P_{\efg}$.  The proof by symmetry is contained in the observation that the six paths are automorphic under the automorphism group.  We show the computational proof in order to demonstrate that the switching permutations are also a single orbit of $\Aut P_{\efg}$.  We compute the nontrivial actions on one of the switching permutations:
\begin{equation}
\begin{aligned}{}
[ \zeta_{\{{jm},{kl},{ij},{mk}\}} (ilm) ]^{(ijk)} &= \zeta_{\{{km},{il},{jk},{mi}\}} (jlm), \\
[ \zeta_{\{{jm},{kl},{ij},{mk}\}} (ilm) ]^{(kji)} &= \zeta_{\{{im},{jl},{ki},{mj}\}} (klm), \\
[ \zeta_{\{{jm},{kl},{ij},{mk}\}} (ilm) ]^{(jk)(lm)} &= \zeta_{\{{kl},{jm},{ik},{lj}\}} (mli), \\
[ \zeta_{\{{jm},{kl},{ij},{mk}\}} (ilm) ]^{(ij)(lm)} &= \zeta_{\{{il},{km},{ji},{lk}\}} (mlj), \\
[ \zeta_{\{{jm},{kl},{ij},{mk}\}} (ilm) ]^{(ik)(lm)} &= \zeta_{\{{jl},{im},{kj},{li}\}} (mlk).
\label{E:P32borbit}
\end{aligned}
\end{equation}
This displays all six switching permutations of $P_{\efg}$.

\emph{In Case (c)},  $X = N[v]$, $\efg = M_{3(m)} := \big\{v_{ij}v_{kl}: \{i,j,k,l\} = \5 \setm m \big\}$, and $\Aut P_{\efg} = \fS_{\5\setm m}$.  The complement of $V(M_{3(m)})$ is $X_m$.  Any vertex in $X_m$ can be taken as $v$; choosing $v = v_{im}$, $\zeta_{N[v_{im}]} (im)$ is a switching automorphism of $P_{3,3}$.  This is the only way to switch $P_{\efg}$ for a switching automorphism, so 
\begin{equation}
\SwAut P_{3,3} = \fS_{\5\setm m} \ \cup \bigcup_{i \in \5 \setm m} \zeta_{N[v_{im}]} (im) \fS_{\5\setm m}.
\label{E:P33setgen}
\end{equation}
Therefore, we may rewrite Equation \eqref{E:P33setgen} as
$$
\SwAut P_{3,3} \ = \bigcup_{\alpha \in \fS_{\5\setm m}} [\zeta_{N[v_{im}]} (im)]^\alpha \ \fS_{\5\setm m}.
$$
(where $i \neq m$ is fixed).

\begin{figure}[hbt]
\begin{center}
\includegraphics[scale=.6]{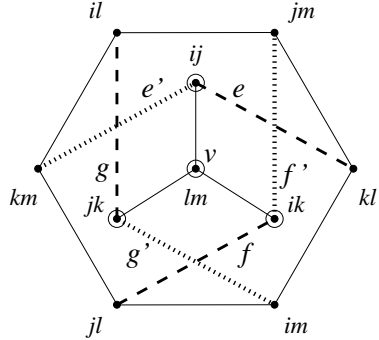}
\caption{Switching the closed neighborhood $X=N[v]$ of a totally positive vertex in $P_{3,3}$ for Case (c).  The original negative edges $e,f,g$ are dashed; the new ones after switching, $e',f',g'$, are dotted.  The heavy lines show the cut $\del X$.}
\label{F:sw4star}
\end{center}
\end{figure}

Much as with $P_{3,2}$, the switching permutations and switching automorphisms of $P_{3,3}$ are whole orbits of the actions of $\Aut P_{\efg}$ on switching permutations and switching automorphisms of $P$.  This is obvious both pictorially, as $\Aut P_{\efg}$ permutes $\5 \setm \{m\}$ and therefore $X_m$, and computationally, as $N[v_{im}]^\alpha = N[v_{i^\alpha m}]$ so $[ \zeta_{N[v_{im}]} (im) ]^\alpha = \zeta_{N[v_{i^\alpha m}]} (i^\alpha m)$.

\subsubsection{The structure of $\SwAut P_{3,2}$.}  \label{structureP32}

A switching automorphism of $P_{3,2}$, if not an automorphism, falls under Case (a) or Case (b).  Thus, 
\begin{equation}
\begin{aligned}
\SwAut P_{3,2} = \Aut P_{3,2} \ &\cup \bigcup_{\lambda \in \langle(ijk)\rangle} [\zeta_{\{jm,kl\}} (jm)(kl)]^\lambda \Aut P_{3,2} \\ 
&\cup \bigcup_{\mu \in \Aut P_{3,2}} [\zeta_{\{{jm},{kl},{ij},{mk}\}} (il)(jk)]^\mu \Aut P_{3,2} .
\end{aligned}
\label{E:P32setgen}
\end{equation}
That tells us the set $\SwAut P_{3,2}$ but to know the group we need the rules for multiplication and for how to determine, for each permutation in $p_A(\SwAut P_{3,2})$, which switching must be done before the permutation to get a switching automorphism.

The description is simplified if we fix the $P_{3,2}$ by choosing a specific negative edge set.  Our choice for $E^-$ is $\{e=v_{14}v_{25},\ f=v_{34}v_{15},\ g=v_{24}v_{35}\} \subseteq H_{45}$.  (That is, we are setting $i,j,k = 1,2,3$ and $l,m = 4,5$.)  

To describe the group we fix two switching sets, 
$$
W := \{v_{15},v_{24}\} \quad \text{ and } \quad Z := \{v_{34},v_{25},v_{13},v_{24}\},
$$  
and corresponding switching permutations, 
$$
\upsilon_W := \zeta_W (15)(24) \quad \text{ and } \quad \omega_Z := \zeta_Z (145).
$$  
($W$ is the $X = \{v_{im},v_{jl}\}$ of Case (a) and $Z$ is the $X = \{v_{jm}, v_{kl}, v_{ij}, v_{km}\}$ of Case (b).  The permutation part is what was called $\gamma_W$ and $\gamma_Z$; as before, it is partly arbitrary since it is determined only up to right multiplication by elements of $\Aut P_{3,2}$.)  
For the systems of representatives in Proposition \ref{P:genmult} we choose
$$
R := \{ \epsilon\,\id \} \cup \{ \upsilon_W^\lambda : \lambda \in \langle(123)\rangle \} \cup \{ \omega_Z^\mu : \mu\in\Aut P_{3,2} \} ,
$$
which we may do because the coset representatives constitute three orbits of $\Aut P_{3,2}$ as shown in Section \ref{typeP3d}, 
and $\bar R := \{\bar\zeta_X\gamma_X : \zeta_X\gamma_X \in R\}$.  As in Cases (a) and (b), $W^{(12)(45)} = W$ and $\rho_X^\mu = \rho_{X^\mu}$ for any $\rho_X = \zeta_X\gamma_X \in R$ and $\mu \in \Aut P_{3,2}$, so $R$ is closed under the action of $\Aut P_{3,2}$.  The sets $W^\mu$ and $Z^\mu$ are found in Table \ref{Tb:P32WX}.  

\begin{table}[hbdp]
\begin{center}
\begin{tabular}{|c||c|c||c|c|}
\hline
\vstrut{15pt}
$\lambda$	&$W^\mu$	&$\upsilon_{W^\mu}$	&$Z^\mu$	&$\omega_{Z^\mu}$	
\\[6pt]
\hline\hline
\vstrut{18pt}
$\id$	&	$W = \{v_{15},v_{24}\}$	&$\zeta_{W} (15)(24)$	
&$Z = \{v_{34},v_{25},v_{13},v_{24}\}$	&$\zeta_{Z} (145)$	
\\[6pt]
\hline
\vstrut{18pt}
$(123)$	&$\{v_{25},v_{34}\}$	&$\zeta_{W}^{(123)} (25)(34)$	
&$\{v_{14},v_{35},v_{12},v_{34}\}$	&$\zeta_{Z}^{(123)} (245)$	
\\[6pt]
\hline
\vstrut{18pt}
$(321)$	&$\{v_{35},v_{14}\}$	&$\zeta_{W}^{(321)} (35)(14)$	
&$\{v_{24},v_{15},v_{23},v_{14}\}$	&$\zeta_{Z}^{(321)} (345)$	
\\[6pt]
\hline\hline
\vstrut{18pt}
$(12)(45)$	&$\{v_{15},v_{24}\}$	&$\upsilon_{W}$	
&$\{v_{35},v_{14},v_{23},v_{15}\}$	&$\zeta_{Z}^{(12)(45)} (542)$	
\\[6pt]
\hline
\vstrut{18pt}
$(23)(45)$	&$\{v_{14},v_{35}\}$	&$\upsilon_{W}^{(321)}$	
&$\{v_{25},v_{34},v_{12},v_{35}\}$	&$\zeta_{Z}^{(23)(45)} (541)$	
\\[6pt]
\hline
\vstrut{18pt}
$(13)(45)$	&$\{v_{34},v_{25}\}$	&$\upsilon_{W}^{(123)}$	
&$\{v_{15},v_{24},v_{13},v_{25}\}$	&$\zeta_{Z}^{(13)(45)} (543)$	
\\[6pt]
\hline
\end{tabular}
\bigskip
\end{center}
\caption{The transforms $W^\mu$ and $Z^\mu$ and associated switching automorphisms, for $\mu \in \Aut P_{3,2}$.  Recall that $\upsilon_W^{\mu}=\upsilon_{W^\mu}$ and $\omega_Z^{\mu}=\omega_{Z^\mu}$.}
\label{Tb:P32WX}
\end{table}

The switching set $X$ associated with $\bar\rho \in \bar R$ is uniquely determined if we insist that $|X| \leq 4$.  (That is how we chose $R$.)  
Thus, we are representing $\SwAut P_{3,2}$ as the disjoint union of the left $\bar R$-cosets of $\Aut P_{3,2}$:
\begin{equation}
\SwAut P_{3,2} = \bigcup_{\zeta_X\gamma_X \in R} \bar\zeta_X\gamma_X \Aut P_{3,2} .
\label{E:P32set}
\end{equation}
Note again that $\zeta_X$ and $-\zeta_X = \zeta_{X^c}$ are equally valid representatives of $\bar\zeta_X$; this fact helps to calculate and interpret the multiplication tables we provide for $\SwAut P_{3,2}$.

The product $(\bar\zeta_X\gamma_X)\alpha \cdot (\bar\zeta_Y\gamma_Y)\beta$ of any two switching automorphisms is completely specified by Proposition \ref{P:genmult}.  To find the product follow this procedure:
\begin{enumerate}
\item Set $Y' = Y^{\alpha\inv}$ and $\zeta_{Y'} = \zeta_Y^{\alpha\inv}$.  Then $\zeta_{Y'}\gamma_{Y'}$ is an element of $R$ because $R$ is closed under the action of $\Aut P_{3,2}$.
\item Calculate $U = X \oplus Y'$ or $(X \oplus Y')^c$, the former if $|X \oplus Y'| \leq 4$ and the latter otherwise.
\item Find $\zeta_U\gamma_U \in R$ to determine $\gamma_U$.
\item The product is $(\bar\zeta_{U}\gamma_U) (\gamma_U\inv \gamma_X\gamma_{Y'}) \cdot \alpha\beta$, which lies in the coset $(\bar\zeta_{U}\gamma_U) \Aut P_{3,2}$.
\item The product $(\zeta_X\gamma_X)\alpha \cdot (\zeta_Y\gamma_Y)\beta$ in $\signs \times \Aut P$, if desired, is $\pm (\zeta_{U}\gamma_U) (\gamma_U\inv \gamma_X\gamma_{Y'}) \cdot \alpha\beta$, with the positive sign if $|X \oplus Y'| \leq 4$ and the negative sign if not.
\end{enumerate}
Steps (2) and (3) can be combined by using Tables \ref{Tb:P32mult-upsi-omega}--\ref{Tb:P32mult-omega}, which give the products of elements $\zeta_X\gamma_X, \ \zeta_{Y'}\gamma_{Y'} \in R$.

\begin{table}[ht]
\begin{center}
\begin{tabular}{|c||c|c|c|}
\hline
\vstrut{15pt}
$\cdot$	
&\hbox to 2cm{\hfill$\upsilon_W$\hfill}	
&\hbox to 2cm{\hfill$\upsilon_W^{(123)}$\hfill}	
&\hbox to 2cm{\hfill$\upsilon_W^{(321)}$\hfill}	
\\[6pt]
\hline\hline
\vstrut{18pt}
$\upsilon_W$	
&$\bareps\,\id$	
&$\omega_Z^{(321)} (123)$
&$\omega_Z^{(13)(45)} (321)$
\\[6pt]
\hline
\vstrut{18pt}
$\upsilon_W^{(123)}$
&$\omega_Z^{(12)(45)} (321)$
&$\bareps\,\id$	
&$\omega_Z (123)$
\\[6pt]
\hline
\vstrut{18pt}
$\upsilon_W^{(321)}$	
&$\omega_Z^{(123)} (123)$
&$\omega_Z^{(23)(45)} (321)$
&$\bareps\,\id$	
\\[6pt]
\hline
\end{tabular}
\bigskip

\begin{tabular}{|c||c|c|c||c|c|c|}
\hline
\vstrut{15pt}
$\cdot$ 
&$\omega_Z$	&$\omega_Z^{(123)}$	&$\omega_Z^{(321)}$	
\\[6pt]
\hline\hline
\vstrut{18pt}
$\upsilon_W$	
&$\omega_Z^{(12)(45)}$	
&$\omega_Z^{(123)} (12)(45)$	
&$\upsilon_W^{(123)} (321)$	
\\[6pt]
\hline
\vstrut{18pt}
$\upsilon_W^{(123)}$	
&$\upsilon_W^{(321)} (321)$	
&$\omega_Z^{(13)(45)}$	
&$\omega_Z^{(321)} (23)(45)$	
\\[6pt]
\hline
\vstrut{18pt}
$\upsilon_W^{(321)}$	
&$\omega_Z (13)(45)$	
&$\upsilon_W (321)$	
&$\omega_Z^{(13)(45)}$	
\\[6pt]
\hline
\end{tabular}
\bigskip

\begin{tabular}{|c||c|c|c|}
\hline
\vstrut{15pt}
$\cdot$ 
&$\omega_Z^{(12)(45)}$	&$\omega_Z^{(23)(45)}$	&$\omega_Z^{(13)(45)}$	
\\[6pt]
\hline\hline
\vstrut{18pt}
$\upsilon_W$	
&$\omega_Z$	
&$-\omega_Z^{(23)(45)} (12)(45)$	
&$\upsilon_W^{(23)(45)} (123)$	
\\[6pt]
\hline
\vstrut{18pt}
$\upsilon_W^{(123)}$	
&$-\omega_Z^{(12)(45)} (23)(45)$	
&$\upsilon_W (123)$	
&$\omega_Z$	
\\[6pt]
\hline
\vstrut{18pt}
$\upsilon_W^{(321)}$	
&$\upsilon_W^{(13)(45)} (123)$	
&$\omega_Z^{(321)}$	
&$-\omega_Z^{(13)(45)} (13)(45)$	
\\[6pt]
\hline
\end{tabular}
\end{center}
\bigskip
\caption{The multiplication table of elements of $\signs \times \Aut P$ that represent coset representatives of the second kind times the 
second and third kinds in $\SwAut P_{3,2}$.}
\label{Tb:P32mult-upsi-omega}
\end{table}

\begin{table}[hbt]
\begin{center}
\begin{tabular}{|c||c|c|c||c|c|c|}
\hline
\vstrut{15pt}
$\cdot$ 
&$\upsilon_W$	&$\upsilon_W^{(123)}$	&$\upsilon_W^{(321)}$	
\\[6pt]
\hline\hline
\vstrut{18pt}
$\omega_Z$	
&$-\omega_Z^{(12)(45)} (12)(45)$	
&$\upsilon_W^{(123)} (321)$	
&$\omega_Z^{(13)(45)}$	
\\[6pt]
\hline
\vstrut{18pt}
$\omega_Z^{(123)}$
&$\omega_Z^{(23)(45)}$	
&$-\omega_Z^{(13)(45)} (23)(45)$	
&$\upsilon_W^{(321)} (321)$	
\\[6pt]
\hline
\vstrut{18pt}
$\omega_Z^{(321)}$	
&$\upsilon_W (321)$	
&$\omega_Z^{(12)(45)}$
&$-\omega_Z^{(23)(45)} (13)(45)$	
\\[6pt]
\hline\hline
\vstrut{18pt}
$\omega_Z^{(12)(45)}$	
&$-\omega_Z (12)(45)$	
&$\omega_Z^{(321)}$	
&$\upsilon_W^{(321)} (123)$	
\\[6pt]
\hline
\vstrut{18pt}
$\omega_Z^{(23)(45)}$	
&$\omega_Z^{(123)}$	
&$\upsilon_W^{(123)} (123)$	
&$-\omega_Z^{(321)} (13)(45)$	
\\[6pt]
\hline
\vstrut{18pt}
$\omega_Z^{(13)(45)}$	
&$\upsilon_W (123)$	
&$-\omega_Z^{(123)} (23)(45)$	
&$\omega_Z$	
\\[6pt]
\hline
\end{tabular}
\bigskip
\end{center}
\caption{The multiplication table of elements of $\signs \times \Aut P$ that represent coset representatives of the third kind times the second kind in $\SwAut P_{3,2}$.}
\label{Tb:P32mult-omega-upsi}
\end{table}

\begin{table}[hbt]
\begin{center}
\begin{tabular}{|c||c|c|c||c|c|c|}
\hline
\vstrut{15pt}
$\cdot$ 
&$\omega_Z$	&$\omega_Z^{(123)}$	&$\omega_Z^{(321)}$	
\\[6pt]
\hline\hline
\vstrut{18pt}
$\omega_Z$	
&$\omega_Z^{(23)(45)}$	&$\upsilon_W$	&$-\upsilon_W^{(321)}(13)(45)$	
\\[6pt]
\hline
\vstrut{18pt}
$\omega_Z^{(123)}$
&$-\upsilon_W(12)(45)$		&$\omega_Z^{(12)(45)}$	&$\upsilon_W^{(123)}$	
\\[6pt]
\hline
\vstrut{18pt}
$\omega_Z^{(321)}$	
&$\upsilon_W^{(321)}$	&$-\upsilon_W^{(123)}(23)(45)$	&$\omega_Z^{(13)(45)}$	
\\[6pt]
\hline\hline
\vstrut{18pt}
$\omega_Z^{(12)(45)}$	
&$-\omega_Z^{(23)(45)}(12)(45)$	&$\bareps\,\id$	&$-\omega_Z^{(13)(45)}(23)(45)$	
\\[6pt]
\hline
\vstrut{18pt}
$\omega_Z^{(23)(45)}$	
&$\bareps\,\id$	&$\omega_Z^{(12)(45)}(12)(45)$	&$\omega_Z^{(13)(45)}(13)(45)$	
\\[6pt]
\hline
\vstrut{18pt}
$\omega_Z^{(13)(45)}$	
&$-\omega_Z^{(23)(45)}(13)(45)$	&$-\omega_Z^{(12)(45)}(23)(45)$	&$\bareps\,\id$	
\\[6pt]
\hline
\end{tabular}
\bigskip

\begin{tabular}{|c||c|c|c||c|c|c|}
\hline
\vstrut{15pt}
$\cdot$ 
&$\omega_Z^{(12)(45)}$	&$\omega_Z^{(23)(45)}$	&$\omega_Z^{(13)(45)}$	
\\[6pt]
\hline\hline
\vstrut{18pt}
$\omega_Z$	
&$-\omega_Z^{(123)}(12)(45)$	&$\bareps\,\id$	&$-\omega_Z^{(321)}(13)(45)$	
\\[6pt]
\hline
\vstrut{18pt}
$\omega_Z^{(123)}$
&$\bareps\,\id$	&$-\omega_Z(12)(45)$	&$-\omega_Z^{(321)}(23)(45)$	
\\[6pt]
\hline
\vstrut{18pt}
$\omega_Z^{(321)}$	
&$-\omega_Z^{(123)}(23)(45)$	&$-\omega_Z(23)(45)$	&$\bareps\,\id$	
\\[6pt]
\hline\hline
\vstrut{18pt}
$\omega_Z^{(12)(45)}$	
&$\omega_Z^{(123)}$	&$\upsilon_W^{(12)(45)}$	&$-\upsilon_W^{(13)(45)}(23)(45)$		
\\[6pt]
\hline
\vstrut{18pt}
$\omega_Z^{(23)(45)}$	
&$-\upsilon_W^{(12)(45)}(12)(45)$	&$\omega_Z$	&$\upsilon_W^{(13)(45)}$	
\\[6pt]
\hline
\vstrut{18pt}
$\omega_Z^{(13)(45)}$	
&$\upsilon_W^{(12)(45)}$	&$-\upsilon_W^{(23)(45)}(13)(45)$		&$\omega_Z^{(321)}$	
\\[6pt]
\hline
\end{tabular}
\end{center}
\bigskip
\caption{The multiplication table of elements of $\signs \times \Aut P$ that represent coset representatives of the third kind in $\SwAut P_{3,2}$.}
\label{Tb:P32mult-omega}
\end{table}

To illustrate the calculations involved in preparing the multiplication tables for $P_{3,2}$ we solve three representative cases.

\begin{exam}\label{X:P32a}
For the first two examples we compute the product of $\omega_Z$ times two other switching permutations in $R$.  First, 
\begin{align*}
\omega_Z \omega_Z 
&= \zeta_{\{{34},{25},{13},{24}\}}(145) \cdot \zeta_{\{{34},{25},{13},{24}\}}(145) 	\\
&= \zeta_{\{{34},{25},{13},{24}\}} \zeta_{\{{34},{25},{13},{24}\}^{(145)\inv}} (145) (145) 	\\
&= \zeta_{\{{34},{25},{13},{24}\}} \zeta_{\{{31},{24},{53},{21}\}} (541) 	
= \zeta_{\{{34},{25},{13},{24}\} \oplus \{{31},{24},{53},{21}\}} (541) 	\\
&= \zeta_{\{{34},{25},{53},{21}\}} (541) 	
= \zeta_{Z^{(23)(45)}} (541) 	
= \omega_{Z}^{(23)(45)} .
\end{align*}

Next, a more complicated example involving complementation of the switching set and a residual permutation that is an automorphism of $P_{3,2}$.
\begin{align*}
\omega_Z \omega_Z^{(321)} 
&= \zeta_{\{{34},{25},{13},{24}\}}(145) \cdot \zeta_{\{{24},{15},{23},{14}\}}(345) 	\\
&= \zeta_{\{{34},{25},{13},{24}\}} \zeta_{\{{24},{15},{23},{14}\}^{(145)\inv}} (145)(345) 	\\
&= \zeta_{\{{34},{25},{13},{24}\}} \zeta_{\{{21},{54},{23},{51}\}} (15)(34) 	
= \zeta_{\{{34},{25},{13},{24}\} \oplus \{{12},{45},{23},{15}\}} (15)(34) 	\\
&= -\zeta_{\{14,35\}} (15)(34) 	
= [-\zeta_{\{14,35\}} (14)(35)] \cdot [(14)(35)]\inv(15)(34)	\\
&= -\upsilon_{\{14,35\}} \cdot (35)(14)(15)(34)	
= -\upsilon_{W}^{(321)} (13)(45).
\end{align*}
\end{exam}

\begin{exam}\label{X:P32b}
We use Example \ref{X:P32a} to compute left multiplication by a transform of $\omega_Z$.  
\begin{align*}
\omega_Z^{(321)} \omega_Z^{(123)} 
&= \big[ \omega_Z \omega_Z^{(123)(321)\inv} \big]^{(321)} 	
= \big[ \omega_Z \omega_Z \big]^{(321)} 	,
\intertext{which by Example \ref{X:P32a}}
&= \big[ -\upsilon_{W}^{(321)} (13)(45) \big]^{(321)} = -\upsilon_{W}^{(123)} (32)(45) .
\end{align*}
\end{exam}

By explicitly inverting the isomorphism $\barp_A : \SwAut P_{3,2} \to \fA_5: \bar\zeta\xi \mapsto \xi$ we can say, for any $\xi \in \fA_5$, exactly which switching function $\zeta_X\gamma_X$ should be associated with it.  

\begin{prop}\label{P:P32aut-sw}
For a permutation $\xi \in \fA_5$, the corresponding switching automorphism of $P_{3,2}$ is $\bar\zeta_X\xi \bar\zeta_X\gamma_X\alpha \in \bar\zeta_X\gamma_X \Aut P_{3,2}$ where $\zeta_X\gamma_X \in R$ is given by
$$
\zeta_X\gamma_X = \begin{cases}
\zeta_\eset \,\id = \eps\,\id	&\text{ if } \{4,5\}^{\xi\inv} = \{4,5\}, \\
\zeta_{\{{34},{25},{13},{24}\}^\lambda}(i45)	&\text{ if } \{4,5\}^{\xi\inv} = \{i,4\}, \text{ where } \lambda = (123)^{i-1}, \\
\zeta_{\{{25},{34},{12},{35}\}^\lambda}(54i)	&\text{ if } \{4,5\}^{\xi\inv} = \{i,5\}, \text{ where } \lambda = (123)^{i-1}, \\
\zeta_{\{i5,j4\}} (i5)(j4) 	&\text{ if } \{4,5\}^{\xi\inv} = \{i,j\} \subset \3, \text{ where } j = i^{(123)} ,
\end{cases}
$$
and $\alpha = \gamma_X\inv \xi$.
\end{prop}

\begin{proof}
The question is to find the vertex set $X$ such that $\gamma_X$, of $\zeta_X\gamma_X \in R$, satisfies $\gamma_X\alpha = \xi$ for some $\alpha \in \Aut P_{3,2}$; in other words, $\gamma_X\inv\xi = \alpha \in \Aut P_{3,2}$.  By this definition of $\alpha$, $\{4,5\}^{\xi\inv} = \{4,5\}^{\alpha\inv\gamma_X\inv}$.  But $\{4,5\}$ is invariant under $\Aut P_{3,2}$.  Therefore, $\{4,5\}^{\xi\inv} = \{4,5\}^{\gamma_X\inv}$, which depends only on the coset of $\Aut P_{3,2}$ to which $\xi$ belongs.  In other words, we need only consider the case $\alpha = \id$, which means we examine only all $\xi = \gamma_X$.  Now the proposition follows easily by inspection of the ten cases of $\gamma_X$.

A better method is to show that the proposition for one $X$ implies it for all $X^\lambda$.  
Replacing $X$ by $X^\lambda$,
$$
\{4,5\}^{(\gamma_{X^\lambda})\inv} 
= \{4,5\}^{(\gamma_{X}^\lambda)\inv} 
= \{4,5\}^{\lambda\inv(\gamma_{X})\inv\lambda}
= (\{4,5\}^{(\gamma_{X})\inv})^\lambda.
$$
Taking $\lambda = (123)^p$, and supposing that $\{4,5\}^{\gamma_X\inv} = \{4,5\}$, $\{i,4\}$, $\{i,5\}$, or $\{i,i^{(123}\}$, we deduce that $\{4,5\}^{(\gamma_{X^\lambda})\inv} = \{4,5\}$, $\{i^\lambda,4\}$, $\{i^\lambda,5\}$, or $\{i^\lambda,(i^\lambda)^{(123)}\}$, respectively.  That proves the claim for $\lambda = (123)^p$.  Thus, we need only check the proposition's validity for $X = \eset, Z, Z^{(12)(45)}, \text{ and } W$, which is easier than checking all ten $X$'s.
\end{proof}

A natural question is whether $\SwAut P_{3,2}$ can be written as a product of subgroups, $\fH \cdot \Aut P_{3,2}$ where $\fH \cap \Aut P_{3,2} = \{\id\}$, or in other words whether there exists a system of left coset representatives that is a subgroup.  It does not, for it is known that no subgroup of $\fA_5$ of order 6 has such a complementary subgroup.

\subsubsection{The structure of $\SwAut P_{3,3}$.}  \label{structureP33}

We know the set $\SwAut P_{3,3}$ but for a full description we need the rule of multiplication and the rule for inverting the projection $p_A$.  It is easier to do this if we fix $m$, so we assume $m=5$.  Then $E^- = M_{3(5)}$, $\Aut P_{3,3} = \fS_{\4}$, and 
\begin{equation}
\SwAut P_{3,3} = \fS_{\4} \cup \bigcup_{j=1}^{4} \bar\zeta_{N[j5]} (j5) \fS_{\4}. 
\label{E:P33set}
\end{equation}
An element of the group has the form $\beta$ or $\bar\zeta_{N[j5]} (j5) \beta$ for $\beta \in \fS_{\4}$ and $j \in \4$.  To compute a product refer to Table \ref{Tb:P33mult}.

\begin{table}[hbt]
\begin{center}
\begin{tabular}{|c||c|c|}
\hline
\vstrut{15pt}
Left $\cdot $Top	&$\beta$		&$\bar\zeta_{N[j5]} (j5) \beta$\\[6pt]
\hline\hline
\vstrut{18pt}
$\alpha$	&$\alpha\beta$	&$\bar\zeta_{N[j^{\alpha\inv}5]} (j^{\alpha\inv}5) \alpha\beta$\\[6pt]
\hline
\vstrut{30pt}
$\bar\zeta_{N[i5]} (i5) \alpha$	&$\bar\zeta_{N[i5]} (i5) \alpha\beta$	
&$\begin{cases} 
\hfill	\alpha\beta &\text{ if } j = i^\alpha \\[5pt]
\bar\zeta_{N[j^{\alpha\inv}5]} (i j^{\alpha\inv}5) \alpha\beta &\text{ if } j \neq i^\alpha
\end{cases}$\\[20pt]
\hline
\end{tabular}
\end{center}
\caption{The multiplication table of $\SwAut P_{3,3}$ with negative edge set $M_{3(5)} = \{v_{12}v_{34},v_{13}v_{24},v_{14}v_{23}\}$.  $i,j \in \4$ and $\alpha, \beta \in \fS_{\4}$.}
\label{Tb:P33mult}
\end{table}

The second product column in Table \ref{Tb:P33mult} requires proof, for which the main step is this computation (done for a switching permutation $\zeta\alpha$ and consequently the same for the switching automorphism $\bar\zeta\alpha$):
\begin{align*}
\alpha \cdot \zeta_{N[j5]} (j5) &= \zeta_{N[j^{\alpha\inv}5]} \alpha (j5) 
	= \zeta_{N[j^{\alpha\inv}5]} (j^{\alpha\inv}5) \cdot \alpha.
\end{align*}
That gives the first product.  For the second we continue the calculation, first when $j=i^\alpha$: 
\begin{align*}
\zeta_{N[i5]} (i5) \alpha \cdot \zeta_{N[i^\alpha5]} (i^\alpha5) 
	&= \zeta_{N[i5]} (i5) \zeta_{N[i5]} (i5) \alpha	
	= \zeta_{N[i5]} \zeta_{N[5i]} (i5) (i5) \alpha	
	= \alpha ;
\end{align*}
second when $j \neq i^\alpha$: 
\begin{align*}
\zeta_{N[i5]} (i5) \alpha \cdot \zeta_{N[j5]} (j5) 
	&= \zeta_{N[i5]} (i5) \zeta_{N[j^{\alpha\inv}5]} (j^{\alpha\inv}5) \alpha	\\
	&= \zeta_{N[i5]} \zeta_{N[j^{\alpha\inv}i]} (i5) (j^{\alpha\inv}5) \alpha	
	= -\zeta_{N[j^{\alpha\inv}5]} (i j^{\alpha\inv} 5)\alpha ,
\end{align*}
because ${N[pq]} \oplus {N[qr]} ={N[pr]^c}$, whence $\zeta_{N[pq]} \zeta_{N[qr]} = \zeta_{N[pr]^c} = -\zeta_{N[pr]}$.

Every permutation $\xi \in \fS_{\5}$ is the projection of a unique element $\bar\zeta_X\gamma_X\cdot\alpha \in \SwAut P_{3,3}$ belonging to the coset $\bar\zeta_X\gamma_X\Aut P_{3,3}$.  The following formulas give $\zeta_X$, $\gamma_X$, and $\alpha$ in terms of $\xi$, thereby inverting $p_A$.  Let $\zeta_X\gamma_X\alpha := p_A\inv(\xi)$.  Then $\zeta_X\gamma_X$ identifies the coset of $\fS_{\4}$, and $\alpha$ identifies the element of $\fS_{\4}$ that gives $\xi$.
\begin{align}
(\zeta_X, \gamma_X, \alpha) &= 
\begin{cases}
(\eps, \id, \xi)		&\text{ if $5$ is fixed by } \xi, \\
(\zeta_{N[5^{\xi\inv} 5]}, (5^{\xi\inv} 5), (5^{\xi\inv} 5) \xi)	&\text{ if $5$ is not fixed.} 
\end{cases}
\end{align}
(Note that $(5^{\xi\inv} 5)\xi$ in cycle form is $\xi$ with $5$ deleted from whichever cycle it is in.  Also note that if we interpret $N[kk]$ as the empty set, so $\zeta_{N[kk]}$ is $\epsilon$, and $(55)$ as the trivial cycle $(5)$, then the first line is subsumed in the second line.)

\subsubsection{The end of the proof}\label{endautproof}

That concludes the proof of Theorem \ref{T:aut}.
\end{proof}

\subsection{Orbits and copies}\label{orbit}

There are two ways signed graphs $\Sigma$ and $\Sigma'$ based on the same graph $\Gamma$ can be isomorphic.  They may have the same set of positive circles, which (by Lemma \ref{L:switching}) is the same as saying they are switching equivalent; then for many purposes they are essentially the same.  The other possibility is that they belong to different switching equivalence classes; in other words, their positive circles are not the same ones even though they correspond under an automorphism of $\Gamma$.  From the automorphism and switching automorphism groups we can deduce the number of signatures of $\Gamma$ that are isomorphic to $\Sigma$ and also the number that are switching inequivalent to $\Sigma$ and to each other, i.e., the number of switching equivalence classes of signatures isomorphic to $\Sigma$.

There is a nice bonus to this: we get an interpretation of the part of $\Aut \Gamma$ that does not belong to $\barp_A(\SwAut\Sigma)$.  Apply any automorphism $\gamma \in \Aut\Gamma$ to $\Sigma$.   
Then $\Sigma^\gamma \sim \Sigma$ if and only if $\gamma \in \barp_A(\SwAut\Sigma)$.  That means $\Sigma^\gamma$ for $\gamma \notin \barp_A(\SwAut\Sigma)$, while isomorphic to $\Sigma$, belongs to a different switching equivalence class.

A fine example is $\SwAut P_{3,2}$, whose projection is the alternating group $\fA_5$.  Any single transposition changes $(P,\sigma) \cong P_{3,2}$ to an inequivalent $(P,\sigma')$, but there is one that is simplest.  In the notation of Table \ref{Tb:autexact}, it is $(lm)$.  This permutation preserves the hexagon $H_{lm}$ that contains $E^-$ while reversing the signs of the hexagon's edges.  Whether there are such distinguished permutations to change one switching automorphism class of $P_1$, $P_{2,2}$, or $P_{2,3}$ to another is not known.

The number of different isomorphic (but possibly switching equivalent) copies of a particular signature $\Sigma$ is the number of orbits of $\Aut\Sigma$, which equals $|\!\Aut\Gamma| / |\!\Aut\Sigma|$.  The number of different copies that are not switching equivalent, i.e., the number of switching equivalence classes of signatures isomorphic to $\Sigma$, is $|\!\Aut\Gamma| / |\!\SwAut\Sigma|$, the number of orbits of $\SwAut\Sigma$.  For instance, $|\!\Aut P_1| = |\!\SwAut P_1| = |\fD_4| = 8$; $|\!\Aut P| / |\!\Aut P_1| = |\!\Aut P| / |\!\SwAut P_1| = 5!/8 = 15$; and (obviously) there are $|E| = 15$ ways to have one negative edge, none of which is switching equivalent to any other.

\begin{table}[htb]
\begin{center}
\begin{tabular}{|c||c|c|c|c|c|c|}
\hline
\vstrut{15pt}$(P,\sigma)$	&$+P$, $-P$ &$P_1$, $-P_1$ &$P_{2,2}$, $-P_{2,2}$ &$P_{2,3}$, $-P_{2,3}$ &$P_{3,2}$, $-P_{3,2}$ &$P_{3,3}$, $-P_{3,3}$	\\[3pt]
\hline
\vstrut{15pt}\# copies	
&$1$	&$15$	&$60$	&$15$	&$20$	&$5$	\\[2pt]
\hline
\vstrut{15pt}\# [copies]	
&$1$	&$15$	&$30$	&$15$	&$2$	&$1$	\\[2pt]
\hline
\end{tabular}
\end{center}
\bigskip
\caption{The number of different signatures of $P$ that are isomorphic to each minimal signed Petersen graph and its negative (`copies'); and the number of switching equivalence classes of such signatures (`[copies]').}
\label{Tb:count}
\end{table}
%

\section{Coloring}\label{col}

A \emph{coloration} (in full, \emph{proper $k$-coloration}, where $k \geq 0$) of a signed graph is a function $\kappa: V \to \{0, \pm1, \pm2, \ldots, \pm k\}$ such that if $vw$ is an edge, then $\kappa(w) \neq \sigma(vw)\kappa(v)$.  
The \emph{chromatic number} $\chi(\Sigma)$ is the smallest $k$ such that there is a proper $k$-coloration of $\Sigma$.  
A signed graph has a second chromatic number, the \emph{zero-free chromatic number} $\chi^*(\Sigma)$; it is the smallest $k$ such that there is a proper $k$-coloration of $\Sigma$ that does not use the color $0$.  As the color $0$ can be replaced by $+(k+1)$ to turn a coloration into a zero-free coloration, $\chi^*(\Sigma) = \chi(\Sigma) + 0$ or $1$.

The chromatic numbers pair with chromatic polynomials.  The \emph{chromatic polynomial} of $\Sigma$ is the function $\chi_\Sigma(2k+1) :=$ the number of proper $k$-colorations, and the \emph{zero-free chromatic polynomial} is $\chi^*_\Sigma(2k) :=$ the number that are zero free.  
(One can prove these functions are monic polynomials of degree $|V|$ by any method that establishes the chromatic polynomial $\chi_\Gamma(y)$ of an ordinary graph; see \cite{SGC}.  There is another connection: $\chi_\Gamma(y) = \chi_{+\Gamma}(y) = \chi^*_{+\Gamma}(y)$.)

\begin{prop}\label{P:chromaticsw}
The chromatic numbers and the chromatic polynomials of a signed graph are invariant under switching and isomorphism.
\end{prop}

\begin{proof}
Isomorphism invariance is obvious.  For switching invariance, consider a proper coloration $\kappa$.  A switching function $\zeta$ acts on $\kappa$ by transforming it to $\kappa^\zeta(v) := \zeta(v)\kappa(v)$.  The condition for a coloration to be proper, $\kappa(w) \neq \kappa(v)\sigma(vw)$, when multiplied by $\zeta(w)$, takes the form 
$$
\kappa^\zeta(w) = \kappa(w)\zeta(w) \neq \kappa(v)\sigma(vw)\zeta(w) = [\kappa(v)\zeta(v)] [\zeta(v)\sigma(vw)\zeta(w)] = \kappa^\zeta(v)\sigma^\zeta(vw).
$$  
Thus, $\kappa^\zeta$ is a proper coloration of $\Sigma^\zeta$ if and only if $\kappa$ is a proper coloration of $\Sigma$.  This establishes a bijection between proper colorations of $\Sigma$ and of $\Sigma^\zeta$ and hence the proposition.
\end{proof}

\subsection{Chromatic numbers}\label{chronum}

The chromatic numbers are weak invariants; they are nearly the same for all signatures of $P$.

\begin{thm}\label{T:col}
The chromatic and zero-free chromatic numbers of signed Petersen graphs are as in Table \ref{Tb:col}.
\end{thm}

\begin{table}[htb]
\begin{center}
\begin{tabular}{|r||c|c|c|c|c|c|}
\hline
$(P,\sigma)$	\vstrut{15pt}&\hbox to 2em{\,$+P$} &\hbox to 2em{\;\;$P_1$} &\hbox to 2em{\,$P_{2,2}$} &\hbox to 2em{\,$P_{2,3}$} &\hbox to 2em{\,$P_{3,2}$}	&\hbox to 2em{\,$P_{3,3}$}	\\[3pt]
\hline
\vstrut{15pt}$\chi(P,\sigma)$	&1	&1	&1	&1	&1	&1	\\
\vstrut{15pt}$\chi^*(P,\sigma)$	&2	&2	&2	&2	&2	&1	\\[2pt]
\hline
\end{tabular}
\end{center}
\bigskip
\caption{The chromatic numbers of signed Petersen graphs.}
\label{Tb:col}
\end{table}

To find the chromatic numbers of any $(P,\sigma)$, switch it into one of the minimal forms and look it up in Table \ref{Tb:col}.  
Note that $+P \simeq -P_{3,3}$, $P_1 \simeq -P_{2,3}$, $P_{2,2} \simeq -P_{2,2}$, $P_{2,3} \simeq -P_1$, $P_{3,2} \simeq -P_{3,2}$, and $P_{3,3} \simeq -P$.

We prepare for the proof of Theorem \ref{T:col} with definitions and a lemma.

By a \emph{signed color} we mean $0$ or $+i$ or $-i$ for $i>0$.  For consistency with the definition of chromatic numbers, when coloring a signed graph we call $\pm1$ a single \emph{unsigned color} and we do not count 0 as an unsigned color.  Thus, the counting of unsigned colors on signed graphs is very different from that on unsigned graphs.  We can color an unsigned graph with signed colors but each has to be counted separately; for example, $0,+1,-1$ are three colors when coloring an unsigned graph.

Note that the endpoints of a negative edge may have the same signed color as long as that color is not 0.  

\emph{Contracting} a graph $\Gamma$ by an edge set $S$ means one shrinks each connected component of the spanning subgraph $(V,S)$ to a vertex.  The contracted graph is written $\Gamma/S$.  
(Technically, a vertex $W$ of $\Gamma/S$ is a subset of $V$ consisting of the vertices of one component of $(V,S)$; they are the vertices that are coalesced into one by the shrinking.)  
The edges of $S$ are deleted.  Another edge becomes a loop if its endpoints belong to the same component of $(V,S)$.  
We say that an original vertex that is a component of $(V,S)$ remains a vertex of $\Gamma/S$.  Any other vertex of $\Gamma/S$ results from coalescing two or more original vertices; we say it \emph{results from contraction} to distinguish it from remaining original vertices.

\begin{lem}\label{L:3sgdcolors}
Let $\Sigma$ be a signed graph and let $m \geq 1$.

\begin{enumerate}[{\rm(a)}]

\item  Suppose $\chi(|\Sigma|/E^-(\Sigma)) \leq 2m$.  Then $\chi(\Sigma) \leq \chi^*(\Sigma) \leq m$.
\label{L:3sgdcolors-m}

\item  Suppose $|\Sigma|/E^-(\Sigma)$ can be colored with the colors $0, \pm1, \ldots, \pm m$ in such a way that no vertex resulting from contraction gets the color $0$.  Then $\chi(\Sigma) \leq m$ and $\chi^*(\Sigma) \leq m+1$.
\label{L:3sgdcolors-m0}

\item If $\chi(|\Sigma|/E^-(\Sigma)) \leq 2$ and $\Sigma$ has at least one edge, then $\chi(\Sigma) =  \chi^*(\Sigma) = 1$.
\label{L:3sgdcolors-2}

\item  If $\chi(|\Sigma|/E^-(\Sigma)) = 3$, then $\chi^*(\Sigma) = 2$.
\label{L:3sgdcolors-3}

\end{enumerate}
\end{lem}

\begin{proof}
(\ref{L:3sgdcolors-m})  Color $|\Sigma|/E^-$ with the colors $\pm1, \ldots, \pm m$.  This coloration can be pulled back to $\Sigma$, because the vertices that are contracted into $W$ can all be given the signed color of $W$.  Thereby we see that $\Sigma$ needs at most $m$ unsigned colors, without using the color 0.

(\ref{L:3sgdcolors-m0})  Color $|\Sigma|/E^-(\Sigma)$ as specified.  This coloration can be pulled back to $\Sigma$, because the vertices that are contracted into $W$ can all be given the signed color of $W$.  Thereby we see that $\Sigma$ needs at most $m$ unsigned colors if $0$ is permitted but it may need $m+1$ if $0$ is excluded.

(\ref{L:3sgdcolors-2})  When the contraction is bipartite, assign color $+1$ to one color class and $-1$ to the other.  Pulling this coloration back to $\Sigma$ yields a zero-free coloration, from which the chromatic numbers follow---as long as there is at least one edge in $\Sigma$ so one cannot color every vertex $0$.

(\ref{L:3sgdcolors-3})  From (\ref{L:3sgdcolors-m}) we conclude that $\chi^*(\Sigma) \leq 2$.  Trying to color $\Sigma$ using only $\pm1$, the endpoints of a negative edge must have the same signed color; therefore, such a coloration of $\Sigma$ can only be a pullback of a 2-coloration of $|\Sigma|/E^-$, which does not exist.  Hence, there is no coloration of $\Sigma$ using only one unsigned color without $0$, and therefore $\chi^*(\Sigma) = 2$.  
\end{proof}

\begin{proof}[Proof of Theorem \ref{T:col}]
The chromatic number of $P$ itself is 3 \cite{TPG}.  Thus, $+P$ needs exactly three signed colors, which may be $0,+1,-1$ if $0$ is used and otherwise must be, for example, $+1,-1,+2$.  

The only bipartite contraction is $P/E^-(P_{3,3})$; it can be colored with $+1,-1$, so $P_{3,3}$ can be colored using $\pm1$.  (One can more easily see this by coloring the switching-isomorphic graph $-P$.)  The other contractions need three or four signed colors.  

$P/E^-(P_{2,d})$ ($d=2,3$) has chromatic number 3, and since there are just two contracted vertices they can get nonzero signed colors; it follows that $P_{2,d}$ is colorable with signed colors $\pm1,0$, no contraction vertex being colored 0.  Therefore, $\chi(P_{2,d}) = 1$ and $\chi^*(P_{2,d}) = 2$.  The same reasoning holds for $P_1$, where there is one contracted vertex.

The most complicated contraction is $P/E^-(P_{3,2})$.  It has a triangle composed of contracted vertices, so its chromatic number is 3 but there does not exist a coloration with colors $\pm1,0$ in which no contracted vertex has color 0.  However, one can color $P_{3,2}$ directly using $\pm1,0$.  The hexagon that contains all negative edges should be colored alternately $+1$ and $0$.  The vertices adjacent to the hexagon get color $-1$ and the remaining vertex is colored $0$ or $+1$.  Thus, $\chi(P_{3,2}) = 1$ and $\chi^*(P_{3,2}) = 2$.
\end{proof}

\subsection{Coloration counts}\label{colcount}

A more refined coloring invariant, the chromatic polynomial, does differ for different signatures of $P$, and most likely the zero-free chromatic polynomials differ as well.  Since the polynomials have degree 10, computing them is too large a project for us.  ($\chi_P(y)$ is known; perhaps it is possible to imitate the technique for calculating it in \cite[Additional Result 12c]{Biggs2}.)  I propose that the number of proper $k$-colorations for any $k \geq 1$, and also the number of zero-free proper $k$-colorations for any $k \geq 2$, is a distinguishing invariant.  We prove this for proper 1-colorations.

\begin{thm}\label{P:chromatic}
Any two signatures of the Petersen graph that are not switching isomorphic have different chromatic polynomials and in particular they have different numbers $\chi_{(P,\sigma)}(3)$ of proper $1$-colorations.
\end{thm}

\begin{conj}\label{Cj:chromatic}
(a)  Two signed Petersen graphs that are not switching isomorphic have different zero-free chromatic polynomials; in particular they have different numbers $\chi_{(P,\sigma)}^*(4)$ of zero-free proper $2$-colorations.  
(b)  For any $\mu \geq 2$, the six values $\chi_{(P,\sigma)}(2\mu+1)$ are different for each switching isomorphism class of sign functions, and so are the six values $\chi^*_{(P,\sigma)}(2\mu)$.
\end{conj}

We will establish Theorem \ref{P:chromatic} by investigating $\chi_{(P,\sigma)}(3) - \chi_{+P}(3)$ with the aid of several general lemmas and formulas.
Calculating the difference give the actual value, because
$$
\chi_{+P}(3) = \chi_P(3) = 120.
$$
A proof depends on the fact that every 3-coloration of $P$ has the same form as every other, under graph automorphisms and permutations of the colors.  In a coloration define a \emph{head vertex} to be a vertex whose neighbors have only one color.  
Each proper $3$-coloration of $P$ has a unique head vertex; and there are $12$ such colorations for each head vertex.  
(To prove this, examine the two ways to 3-color $N[v]$ where $v$ is the head vertex.  We omit the details.)  
To color with a given head vertex, one chooses its color, then chooses the neighborhood color, then colors the uncolored hexagon with the two non-neighborhood colors.  One concludes that $\chi_P(3) = 120$.

We begin preparing for the proof of Theorem \ref{P:chromatic} with the balanced expansion formula of \cite[Theorem 1.1]{CISG}, which states that for any signed graph $\Sigma = (\Gamma,\sigma)$, 
\begin{equation}
\chi_{\Sigma}(2\mu+1) = \sum_{\substack{W\subseteq V:\\ W \text{ independent}}} \chi_{\Sigma\setm W}^*(2\mu).
\label{E:balexp}
\end{equation}
(The proof is easy, by counting colorations according to the set $W$ with color $0$.)  Applying this to the difference of $\Sigma$ and $+\Gamma$,
\begin{align*}
\chi_{\Sigma}(2\mu+1) - \chi_{+\Gamma}(2\mu+1) 
&= \sum_{\substack{W\subseteq V:\\ W \text{ independent}}} \chi^*_{\Sigma\setm W}(2\mu) - \chi^*_{+\Gamma\setm W}(2\mu).
\end{align*}
The term of $W$ disappears if $\Sigma\setm W$ is balanced; thus, 
\begin{align}
\chi_{\Sigma}(2\mu+1) - \chi_{\Gamma}(2\mu+1) 
&= \sum_{\substack{W\subseteq V:\\ W \text{ independent,}\\ \Sigma\setm W \text{ unbalanced}}} \chi^*_{\Sigma\setm W}(2\mu) - \chi_{\Gamma\setm W}(2\mu),
\label{E:coldiff}
\end{align}
since $\chi^*_{+\Gamma}(y) = \chi_{\Gamma}(y)$.  

Observe that
\begin{equation}
\chi^*_\Sigma(2) = 
\begin{cases}
2^{c(\Sigma)}	&\text{ if } \Sigma \text{ is antibalanced}, \\
0		&\text{ if it is not}.
\end{cases}
\label{E:2col}
\end{equation}
To prove this, suppose a zero-free, proper 1-coloration exists.  Since there are only the two signed colors $+1$ and $-1$, a negative edge must have the same color at both ends and a positive edge must have oppositely signed colors at its ends.  Taking the bipartition of $V$ into sets of vertices with the same sign, that means a positive edge in $-\Sigma$ has both ends in the same part and a negative edge has ends in opposite parts.  Hence, $-\Sigma$ is balanced and $\Sigma$ is antibalanced.  If $\Sigma$ is antibalanced, there are two choices of color in each component.

\begin{lem}\label{L:bipartbalanti}
If $\Sigma$ has two of the properties of balance, antibalance, and bipartiteness, then it has the third property as well.
\end{lem}

\begin{proof}
Balance means every circle is positive.  Antibalance means every even circle is positive and every odd circle is negative.  
In a bipartite signed graph, balance and antibalance are equivalent.
In any signed graph, the conjunction of balance and antibalance implies there are no odd circles.  
\end{proof}

Now we can further simplify Equation \eqref{E:coldiff} when $\mu=1$.  
By Lemma \ref{L:bipartbalanti} there are three possibilities:  $\Gamma\setm W$ may be bipartite with $\Sigma\setm W$ not antibalanced, $\Sigma\setm W$ may be antibalanced but nonbipartite, or it may be nonbipartite and not antibalanced.  Then by Equation \eqref{E:2col}, 
\begin{align}
\begin{aligned}
\chi_{\Sigma}(3) - \chi_{\Gamma}(3) 
&= \sum_{\substack{W\subseteq V:\\ W \text{ independent,}\\ \Sigma\setm W \text{ antibalanced and not bipartite}}} 2^{c(\Gamma\setm W)} \\
&\quad - \sum_{\substack{W\subseteq V:\\ W \text{ independent,}\\ \Sigma\setm W \text{ bipartite and not antibalanced}}} 2^{c(\Gamma\setm W)}.
\end{aligned}
\label{E:3diff}
\end{align}

\begin{proof}[Proof of Theorem \ref{P:chromatic}]  
We use a formula deduced from Equation \eqref{E:3diff}.  For $k=0,1,2$, let 
\begin{align*}
\alpha_k(\Sigma) := &\text{ the number of independent sets } X \subseteq V \text{ such that } \Sigma\setm X \text{ is balanced.}
\end{align*}

\begin{lem}\label{L:petdiff}
For a signed Petersen graph, 
\begin{equation}
\chi_{(P,\sigma)}(3) - \chi_{+P}(3) 
= 2\alpha_0(-(P,\sigma)) + 2\alpha_1(-(P,\sigma)) + 2\alpha_2(-(P,\sigma)) - 4 c_6^-(P,\sigma) .
\label{E:petdiffsimp}
\end{equation}
\end{lem}

\begin{proof}
By Section \ref{structure}, either $|W| \leq 1$, $W$ is a pair of nonadjacent vertices, or $W = N(v)$ for some vertex $v$.  In the former cases $P\setm W$ is connected and nonbipartite.  In the last case it is bipartite.

Suppose $(P,\sigma)\setm W$ is antibalanced and not bipartite.  
Because $P\setm W$ is not bipartite, $|W|\leq 2$.  Therefore, $P\setm W$ is connected and the term of $W$ contributes $2$ to the first summation if $(P,\sigma)\setm W$ is antibalanced, $0$ otherwise.  The respective contributions of $W$ of size $0, 1, 2$ are $2\alpha_0(-(P,\sigma))$, $2\alpha_1(-(P,\sigma))$, and $2\alpha_2(-(P,\sigma))$.

Suppose $P\setm W$ is bipartite and not antibalanced.  
Here $W = N(v)$ so $P\setm W = H_v \cupdot K_1$.  Because $(P,\sigma)\setm W$ is not antibalanced, the term of $W$ contributes $4$ to the second summation.  Each hexagon lies in $P\setm W$ for a unique $W = N(v)$.  Since the contribution of each negative hexagon to \eqref{E:petdiffsimp} is $-4$, the total contribution of all negative hexagons is $4 c_6^-(P,\sigma)$.
\end{proof}

It remains to evaluate the $\alpha_k$, as $c_6^-$ is given by Table \ref{Tb:circ}.  The results are in Table \ref{Tb:3diff} along with the values of $\chi_{(P,\sigma)}(3) - \chi_{+P}(3)$ and $\chi_{(P,\sigma)}(3)$.
\begin{table}[hbt]
\begin{center}
\begin{tabular}{|c||c|c|c|c|c|c|}
\hline
$(P,\sigma)$ \vstrut{15pt}&\hbox to 3em{\hfill$+P$\hfill} &\hbox to 3em{\hfill $P_1$\hfill} &\hbox to 3em{\hfill $P_{2,2}$\hfill} &{$P_{2,3} \simeq -P_1$} &\hbox to 3em{\hfill $P_{3,2}$\hfill} &{$P_{3,3} \simeq -P$}	\\[3pt]
\hline\hline
\vstrut{15pt}$\alpha_0(P,\sigma)$	&1	&0	&0	&0	&0	&0	\\[2pt]
\vstrut{15pt}$\alpha_1(P,\sigma)$	&10	&2	&0	&0	&0	&0	\\[2pt]
\vstrut{15pt}$\alpha_2(P,\sigma)$	&30	&14	&6	&4	&0	&0	\\[4pt]
\hline
\vstrut{15pt}$c_6^-(P,\sigma)$		&0	&4	&6	&4	&10	&0	\\[4pt]
\hline\hline
\vstrut{15pt}$\chi_{(P,\sigma)}(3) - \chi_{+P}(3)$	&0	&$-8$	&$-12$	&16	&$-40$	&82	\\[4pt]
\hline
\vstrut{15pt}$\chi_{(P,\sigma)}(3)$	&120	&112	&108	&136	&80	&202	\\[4pt]
\hline
\end{tabular}
\end{center}
\bigskip
\caption{The numbers necessary to prove Theorem \ref{P:chromatic}.}
\label{Tb:3diff}
\end{table}

Some of the values $\alpha_k$ are not obvious.  For $P_{3,2}$ and $-P$, all $\alpha_k = 0$ because $l_0 > 2$ (Theorem \ref{T:frno}).  
$\alpha_1(P_1) = 2$ because any edge is the intersection of two pentagons, hence only by deleting an endpoint of the negative edge can we balance $P_1$.  
$\alpha_1(P_{2,2}) = \alpha_1(P_{2,3}) = 0$ because each graph has $l_0 > 1$.  That leaves $\alpha_2$ of $P_{2,2},$ $P_1$, and $-P_1 \simeq P_{2,3}$.

Consider deleting a nonadjacent vertex pair from $P_{2,2} \simeq -P_{2,2}$.  Suppose the negative edges are $v_{15}v_{34}$ and $v_{23}v_{45}$ (see Figure \ref{F:P}).  We get balance by deleting one endpoint of each edge, ignoring $\{v_{15},v_{23}\}$ because those vertices are adjacent; that is three vertex pairs.  If we switch $v_{15}$ and $v_{23}$ first so the negative edges are $v_{15}v{24}$ and $v_{23}v_{14}$, we find three more ways to get balance.  Thus, the obvious approach gives six balancing sets.  These are all.  To prove that, we list four negative pentagons forming two vertex-disjoint pairs: 
$$A := v_{24}v_{15}v_{34}v_{12}v_{35} \text{ and } A' := v_{14}v_{23}v_{45}v_{13}v_{25},$$ 
and 
$$B := v_{14}v_{23}v_{45}v_{12}v_{35} \text{ and } B' := v_{34}v_{15}v_{24}v_{13}v_{25}.$$
We need one vertex from each pair, which means (Case 1) one from $A \cap B = \{v_{12},v_{35}\}$ and one from $A' \cap B' = \{v_{13},v_{25}\}$, or else (Case 2) one from $A \cap B' = \{v_{24},v_{15},v_{34}\}$ and one from $A' \cap B = \{v_{14},v_{23},v_{45}\}$.  The two other negative pentagons are 
$$C := v_{15}v_{23}v_{45}v_{13}v_{24} \text{ and } D := v_{34}v_{15}v_{23}v_{14}v_{25} .$$
Case 1 cannot cover both of these.  In Case 2, we can take any pair except $v_{24}v_{45}$, $v_{34}v_{14}$, or (because they are adjacent) $v_{15}v_{23}$.  Therefore, $\alpha_2(P_{2,2}) = 6$.

Next, consider $P_1$ with negative edge $v_{15}v_{23}$.  The obvious pairs are $v_{15}$ and any non-neighbor, and $v_{23}$ and any of its non-neighbors; that is 12 pairs.  Two pairs that are less obvious are $\{v_{24},v_{34}\}$ and $\{v_{14},v_{54}\}$, which eliminate all circles on $v_{15}$ and $v_{23}$, respectively.  To show there are no other possible pairs we list the negative pentagons:
$$
D, \ C, \ v_{12}v_{34}v_{15}v_{23}v_{45}, \ v_{14}v_{35}v_{24}v_{13}v_{25}.
$$
If a pair excludes $v_{15}$ and $v_{23}$ it needs one vertex from each of the following triples:
$$
v_{12}v_{34}v_{45}, \ v_{45}v_{13}v_{24}, \ v_{14}v_{35}v_{24}, \ v_{14}v_{25}v_{34},
$$
in which nonconsecutive sets are disjoint.  The possible pairs then are $v_{45}v_{14}$ and $v_{24}v_{34}$.  Thus, $\alpha_2(P_1) = 14.$

Finally, consider $-P_1 \sim P_{2,3}$ with (after switching $X_4$) negative edges $e := v_{12}v_{35}$ and $f := v_{13}v_{25}$.  The obvious pairs are one from $e$ and one from $f$.  They are the only ones possible.  As with $P_{2,2}$, $A, A', B, B'$ are negative and we have two cases.  Case 1 gives the four obvious vertex pairs.  Case 2 is impossible, because it fails to cover every negative pentagon, which is every pentagon that does not contain the edge $v_{15}v_{23}$.  Hence, $\alpha_2(P_{2,3}) = 4$.

The values of $\chi_{(P,\sigma)}(3) - \chi_{+P}(3)$ and $\chi_{(P,\sigma)}(3)$ follow from Lemma \ref{L:petdiff}.  
(I also calculated $\chi_{P_1}(3)$ and $\chi_{-P}(3)$ directly, confirming the values $112$ and $202$.)  
They are different for each switching isomorphism type; that proves the theorem.
\end{proof}

Theorem \ref{P:chromatic} suggests a problem.

\begin{question}\label{Q:chromatic}
Is it possible for two switching-nonisomorphic signatures of the same graph to have the same chromatic polynomial?  Can they have the same zero-free chromatic polynomial?
\end{question}

It is not possible for a 2-regular graph.

\begin{prop}\label{P:2regchromatic}
Two different, switching nonisomorphic signatures of the same $2$-regular graph have different chromatic polynomials and different zero-free chromatic polynomials.
\end{prop}

\begin{proof}
It suffices to consider a circle $C_l$ with two signatures, $\sigma_0$ in which it is positive and $\sigma_1$ in which it is negative.  It is well known that $\chi_{C_l}(y) = (y-1) \big[ (y-1)^{l-1} - (-1)^{l-1} \big]$; thus, 
$$
\chi_{(C_l,\sigma_0)}(y) = \chi^*_{(C_l,\sigma_0)}(y) = (y-1) \big[ (y-1)^{l-1} - (-1)^{l-1} \big] .
$$  

To calculate the polynomials of $\Sigma_1 := (C_l,\sigma_1)$ we apply the matroid theory of \cite{SG, SGC}.  By \cite[Theorem 5.1]{SG} the matroid $G(\Sigma_1)$ is the free matroid $F_l$ on $l$ points, whose characteristic polynomial is $\sum_A (-1)^{|A|}y^{|A|}$, summed over all flats, i.e., all subsets of $E$; thus it equals $(y-1)^l$.  By \cite[Theorem 2.4]{SGC}, $\chi_{\Sigma_1}(y)$ equals the characteristic polynomial of $F_l$.  
For $\chi^*_{\Sigma_1}(y)$ we sum only over balanced sets $A$; since the only unbalanced flat is $E$, $\chi^*_{\Sigma_1}(y) = (y-1)^l - (-1)^l$.
\end{proof}

A possible approach to Question \ref{Q:chromatic} may be through the geometrical interpretation of signed-graph coloring in \cite[Section 5]{IOP}.

\section{Clusterability}\label{clust}

A signed graph $\Sigma$ is called \emph{clusterable} if its vertices can be partitioned into sets, called clusters, so that each edge within a cluster is positive and each edge between two clusters is negative.  Such a partition is a \emph{clustering} of $\Sigma$.  By Proposition \ref{P:balance} balance is clusterability with at most two clusters.  
Clusterability is the other property we discuss, besides the automorphism group, that is not invariant under switching.  
Davis proposed it as a possibly more realistic alternative to balance as an ideal state of a social group \cite{Davis}, and he proved:

\begin{prop}\label{P:clust}
A signed graph is clusterable if and only if no circle has exactly one negative edge.
\end{prop}

Clusterability of signed graphs has recently taken on new life in the field of knowledge and document classification under the name `correlation clustering' \cite{Bansal}.

\newcommand\clun{\operatorname{clu}}

There are (at least) two ways to measure clusterability.  When $\Sigma$ is clusterable, the smallest possible number of clusters is the \emph{cluster number} $\clun(\Sigma)$.  Even if a signed graph is inclusterable, it becomes clusterable when enough edges are deleted; the smallest such number is the \emph{inclusterability index} $Q(\Sigma)$.  

\begin{thm}\label{T:clcontraction}
The cluster number of a signed graph is $\clun(\Sigma) = \chi(|\Sigma|/E^+(\Sigma))$.  $\Sigma$ is clusterable if and only if $|\Sigma|/E^+(\Sigma)$ has no loops.  
\end{thm}

Thus an all-positive signed graph is a cluster by itself: $\clun(+\Gamma) = 1$.  For an all-negative signed graph, $\clun(-\Gamma) = \chi(\Gamma)$.

\begin{proof}
In the contraction $\Gamma' := |\Sigma|/E^+$, let $[v] \in V'$ denote the vertex corresponding to $v \in V$.

Suppose $\Sigma$ has a clustering $\pi = \{V_1, \ldots, V_k\}$ into $k$ parts (with each $V_i$ nonempty).  That means, first, that all positive edges are contained within $V_i$'s, so each $[v]$ is contained within a set $V_i$.  Furthermore, two vertices $[u], [v] \in V'$ that lie within the same $V_i$ are nonadjacent, since $E' = E^-$ and no negative edges are within $V_i$.  Therefore the function $\kappa: V \to \{1,2,\ldots,k\}$ defined by $\kappa(v) = i$ if $[v] \subseteq V_i$ is a (proper) coloration of $\Gamma'$, and furthermore every color is used at one or more vertices.  
($\kappa$ is determined by $\pi$ only up to permutations of the colors.)

Conversely, if $\kappa'$ is a (proper) coloration of $\Gamma'$ using exactly $k$ colors, say with color set $\{1,2,\ldots,k\}$, let $V_i := \{v \in V:  \kappa'([v]) = i\}$.  That implies $\Gamma'$ has no loops and that every color is applied to a vertex, so no $V_i$ is empty.  Then in $\Sigma$, no negative edge can lie within a set $V_i$ and, because every positive edge of $\Sigma$ is within a set $[v]$, it lies inside a $V_i$.  Hence, $\pi = \{V_1, \ldots, V_k\}$ is a clustering of $\Sigma$ into $k$ clusters.

Consequently, clusterings of $\Sigma$ coincide (modulo permuting the colors) with $k$-colorations of $\Gamma'$ that use all $k$ colors, for any $k$.  The theorem follows immediately.
\end{proof}

Observe that $|\Sigma|/E^+(\Sigma) = |\Sigma|/E^-(-\Sigma)$.  Thus, the contraction used here in connection with $\Sigma$ is the same one used in Theorem \ref{T:col} in connection with $-\Sigma$.

To supplement Davis's criterion for clusterability---that is, for zero inclusterability index---we state a criterion for unit index.  The proof is a simple check.

\begin{prop}\label{P:clust1}
 $Q(\Sigma) = 1$ if and only if there is a circle with exactly one negative edge and there is an edge common to all such circles.
\end{prop}

\begin{thm}\label{T:clust}
The clusterabilities of the minimal signed Petersen graphs and their negatives are as stated in Table \ref{Tb:clust}.
\end{thm}

\begin{table}[hbt]
\begin{center}
\begin{tabular}{|r||c|c|c|c|c|c|c|c|c|c|c|c|}
\hline
$(P,\sigma)$	\vstrut{15pt}
&\hbox to 2em{\,$+P$}	&\hbox to 2em{\,$-P$}	
&\hbox to 2em{\;\;$P_1$}	&\hbox to 2em{$-P_1$}
&\hbox to 2em{\,$P_{2,2}$}	&\hbox to 2.4em{$-P_{2,2}$}	
&\hbox to 2em{\,$P_{2,3}$}	&\hbox to 2.4em{$-P_{2,3}$}	
&\hbox to 2em{\,$P_{3,2}$}	&\hbox to 2.4em{$-P_{3,2}$}	
&\hbox to 2em{\,$P_{3,3}$}	&\hbox to 2.4em{$-P_{3,3}$}	\\[2pt]
\hline
\vstrut{15pt}$\clun(P,\sigma)$	
&1	&3	
&--	&3	
&--	&3	
&--	&3	
&--	&4	
&--	&2	\\
\vstrut{15pt}$Q(P,\sigma)$	
&0	&0	
&1	&0	
&2	&0	
&2	&0	
&3	&0	
&3	&0	\\[2pt]
\hline
\end{tabular}
\bigskip
\end{center}
\caption{The clusterability measures of the minimal signed Petersen graphs and their negatives.  A dash denotes an inclusterable signature.}
\label{Tb:clust}
\end{table}

\begin{proof}
The cluster numbers are obvious for $+P$, which is balanced, and $P_1, P_{2,2}, P_{2,3}, P_{3,2}, P_{3,3}$, all of which violate Davis's criterion for clusterability.  The negatives of these graphs are clusterable; their cluster numbers follow from Theorem \ref {T:clcontraction}.  Specifically:

The contraction $P/E^+(-P_{2,2})$ has a triangle and is easy to color in 3 colors; thus, $\clun(-P_{2,2}) = 3$.  

The more complex graph $P/E^+(-P_{3,2})$ consists of three triangles overlapping at vertices---which require three colors arranged so that the three divalent vertices have different colors---and one more vertex adjacent to the divalent vertices; therefore, the chromatic number is 4.  That gives $\clun(-P_{3,2}) = 4$.

The contraction $P/E^+(-P_{3,3}) = K_{3,4}$.  Thus, $\clun(-P_{3,3}) = 2$.

The contraction $P/E^+(-P_{2,3})$ is $K_{3,4}$ with one vertex split, forming a $C_5$.  As the contraction is nonbipartite, $\clun(-P_{2,3}) > 2$, but as only one vertex was split, only one more color is needed.

The fact that clusterability is equivalent to having inclusterability index 0 leaves five signatures with positive inclusterability index.  Clearly, $Q(\Sigma) \leq |E^-|$.  That implies $Q(P_1) = 1$.  Proposition \ref {P:clust1} implies that the other inclusterability indices are at least 2, since in each $P_{k,d}$ there are two edge-disjoint circles containing exactly one negative edge each.  Consequently, $Q(P_{2,2}) = Q(P_{2,3}) = 2$.

In each of $P_{3,2}$ and $P_{3,3}$, all the pentagons with one edge on the outer pentagon in Figure \ref{F:3neg-clust} have exactly one negative edge.  Call them the \emph{sharp pentagons}.  To make the signed graph clusterable we must eliminate (at least) all sharp pentagons; thus, we have to remove at least an edge from each one.  Any two sharp pentagons have just one edge in common, and no three of them have a common edge.  Therefore, to eliminate sharp pentagons one has to delete at least three edges.  It follows that $Q(P_{3,2}) = Q(P_{3,3}) = 3$.
\end{proof}

\begin{figure}[htbp]
\begin{center}
\includegraphics[scale=.55]{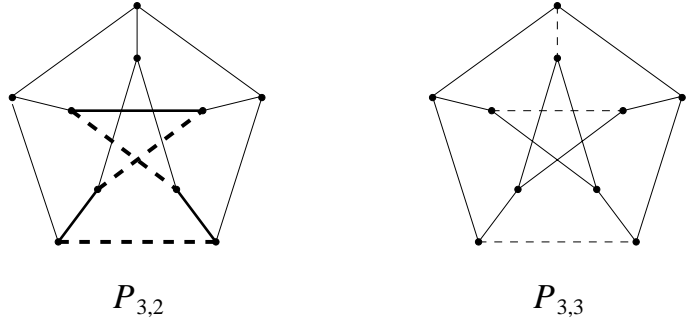}
\caption{Signed Petersen graphs with three negative edges.  Each sharp pentagon has one negative edge.}
\label{F:3neg-clust}
\end{center}
\end{figure}

As clusterability is not a switching invariant, the data in Table \ref{Tb:clust} are not sufficient to describe all signatures of the Petersen graph.  The number of inequivalent clustering problems equals the number of nonisomorphic edge 2-colorations of $P$, which is large.  That makes it interesting to ask about the maximum inclusterability of $P$, defined as the maximum inclusterability index of any signature.

\begin{thm}\label{T:maxclust}
The largest inclusterability index of any signed Petersen graph is $3$.
\end{thm}

\begin{proof}
Several of the signatures in Table \ref{Tb:clust} attain inclusterability 3, so the problem is to prove no higher value is possible.  

We begin with two general observations.  First, every signed graph satisfies
\begin{equation}
\Sigma' \subseteq \Sigma \implies Q(\Sigma') \leq Q(\Sigma).
\label{E:subclust}
\end{equation}
Second, here are properties of general graphs and cubic graphs.

\begin{lem}\label{L:cutclust}
If the underlying graph of a signed graph $\Sigma$ has a cut with more negative than positive edges, then $Q(\Sigma) < |E^-|$.
\end{lem}

\begin{proof}
If there is a cut $\del X$ with more negative than positive edges, delete the positive edges of $\del X$ and any negative edges outside $\del X$.  In the remaining graph $(P,\sigma) \setm S$ the negative edges form a cut, so $(P,\sigma) \setm S$ is clusterable; but as the number of edges that were deleted is less than $|E^-|$, $Q(P,\sigma) < |E^-|$.  
\end{proof}

\begin{prop}\label{P:cubicmaxclust}
Let $\Gamma$ be a graph whose maximum degree is at most $3$.  The maximum inclusterability index of any signature is attained only by signatures in which the negative edge set is a matching.
\end{prop}

\begin{proof}
This follows from Lemma \ref{L:cutclust} by examining the vertex cuts $\del\{v\}$ in a signature that maximizes inclusterability.
\end{proof}

\emph{Proof of Theorem \ref{T:maxclust}, continued.}
We may assume that $(P,\sigma)$ is a signed Petersen that has maximum inclusterability and that $E^-$ is a matching.  A matching in $P$ has at most 5 edges.  The matchings were classified in Section \ref{matchings}.

If $|E^-|$ has 5 edges, it separates two pentagons.  Since $(P,\sigma) \setm E^-$ is all positive, $(P,\sigma)$ is clusterable with two clusters that are the vertex sets of the pentagons of $P \setm E^-$.

Suppose, then, that $E^-$ is a matching with 4 edges.

Lemma \ref{L:cutclust} applies when $E^- = M_5 \setm$ edge, with $X = V(C)$ where $C$ is one of the pentagons separated by $M_5$.

If the matching is $M_4'$, there is a hexagon $H_{lm}$ with three negative edges and the fourth negative edge $d$ is incident with $v_{lm}$ (Figure \ref{F:m3types}).  The two negative edges at distance 2 from $d$, together with $d$, are part of an $M_5$ that is a 5-edge cut with three negative edges.

It follows that $Q(P,\sigma) < 4$ when $E^-$ is a 4-edge matching, so the theorem is proved.
\end{proof}


\section{Other Aspects}\label{other}

The signed Petersen graphs have other properties that we intend to treat elsewhere.  

For instance, we can establish the smallest surface in which each $(P,\sigma)$ can be embedded so that a circle is orientable if and only if it is positive (this is called \emph{orientation embedding}).  This embeddability, by its definition, is a property of switching isomorphism classes, so there are just six cases.  The only signature that embeds in the projective plane is $P_{2,3}$; as $P$ is nonplanar, every other signature of $P$ embeds only in a higher nonorientable surface (if not balanced) or in the torus (if balanced). 

Another aspect is the relationship between $(P,\sigma)$ and its signed covering graph (the `derived graph' of \cite[Section 9]{Biggs2}), in which each vertex of $P$ splits into a pair, $+v$ and $-v$, and edges double as well, with positive edges connecting vertices of the same sign and negative edges connecting vertices of opposite sign.  The switching automorphisms of the signed graph are closely related to the fibered automorphisms of the signed covering.

As $\Aut\Sigma$ is not invariant under switching, there is a very large number of possible automorphism groups of signed Petersen graphs: as many as there are nonisomorphic sets of signatures with negatives paired together.  (We should pair $\Aut(\Sigma)$ with $\Aut(-\Sigma)$ because they have the same automorphisms by Proposition \ref{P:negaut}.)  Sometimes the two members of the pair are isomorphic.  Table \ref{Tb:aut} shows examples.)  The number of such sets is unknown.

Switching and the switching automorphism group generalize from the sign group to any group $\fG$.  A \emph{gain graph} is a graph whose edges are labelled invertibly by elements of $\fG$; this means that, if $\phi(e)$ is the gain of oriented edge $e$ and $e\inv$ is $e$ in the opposite orientation, then $\phi(e\inv) = \phi(e)\inv$.  Gain graphs and switching over arbitrary groups were introduced in \cite{BG1}.  Many of the basic properties of switching automorphisms should extend to the general case, though some, such as the simple description of the switching kernel $\fK$, may depend on having an abelian group, and some (at least, the property that the gain is independent of direction) require a group of exponent 2.  This brief description is just an outline; a complete theory of switching automorphisms over an arbitrary gain group, and its application to examples, are open problems.


\end{document}